\documentclass{amsart}
\usepackage{amscd,amsthm}
\usepackage{latexsym}
\usepackage{capt-of}
\usepackage{graphicx}

\DeclareFontFamily{U}{wncy}{}
\DeclareFontShape{U}{wncy}{m}{n}{<->wncyr10}{}
\DeclareSymbolFont{mcy}{U}{wncy}{m}{n}
\DeclareMathSymbol{\Sha}{\mathord}{mcy}{"58}

\usepackage{amssymb,amsmath}
\usepackage{amsfonts}
\usepackage{tkz-graph}
\newcounter{ctfig}

\newenvironment{myempty}{\refstepcounter{ctfig}}{}
\newcommand{\Z}{{\mathbb Z}}
\newcommand{\Q}{{\mathbb Q}}

\newcommand{\C}{\mathcal{C}}

\newcommand{\A}{\mathbb{A}}
\newcommand{\F}{\mathcal{F}}
\renewcommand{\P}{\mathbb{P}}
\renewcommand{\S}{\mathcal{S}}
\newcommand{\cyc}{\mathcal{C}}
\newcommand{\T}{\mathcal{T}}
\renewcommand{\L}{\mathcal{L}}

\newcommand{\JacC}{{\hbox{Jac}_{\lower.5pt\hbox{$_\C$}}}}
\newcommand{\JacF}{{\hbox{Jac}_{\lower.5pt\hbox{$_\F$}}}}
\newcommand{\etale}{{\hbox{\' et}}}

\newcommand{\tr}{\mathop{\rm tr}}

\newcommand{\Frob}{\mathop{\rm Frob}}
\newcommand{\Gal}{\mathop{\rm Gal}}

\theoremstyle{plain}
\newtheorem{thm}{Theorem}[section]
\newtheorem{conj}[thm]{Conjecture}
\newtheorem{lemma}[thm]{Lemma}
\newtheorem{prop}[thm]{Proposition}
\newtheorem{cor}[thm]{Corollary}
\theoremstyle{definition}
\newtheorem{meth}[thm]{Method}
\newtheorem{example}[thm]{Example}
\newtheorem{remark}[thm]{Remark}
\newtheorem{defn}[thm]{Definition}
\def\Q{{\mathbb Q}}

\def\F{{\mathbb F}}
\def\Z{{\mathbb Z}}
\def\C{{\mathbb C}}
\def\O{{\mathcal O}}
\newcommand{\Pic}{\mathop{\rm Pic}}

\begin{document}
%% \foreach \x in{graphics,floats}{%
%%     \immediate\write18{pdflatex -jobname=template-\x\space "\def\noexpand\placeholder{\x} \noexpand\input{template}"}%
%%     \includepdf[pages=-]{template-\x}%
%% }
\bibliographystyle{plain}
\bibstyle{plain}

\title[Modular forms from $\phi^4$ theory]{New realizations of modular forms in Calabi-Yau threefolds arising from $\phi^4$ theory}
%\title{New realizations of modular forms in Calabi-Yau threefolds arising from $\phi^4$ theory}

\author{Adam Logan}
\email{adam.m.logan@gmail.com}
\address{The Tutte Institute for Mathematics and Computation,
P.O. Box 9703, Terminal, Ottawa, ON K1G 3Z4, Canada}

\date{\today}

\begin{abstract}
Brown and Schnetz found that the number of
points over $\F_p$ of a graph hypersurface is often related to the coefficients
of a modular form.  We set some of the reduction techniques used to discover
such relation in a general geometric context.  We also prove the relation for
one example
of a modular form of weight $4$ and two of weight $3$, refine the
statement and suggest a method of proving it for four more of weight $4$,
and use the one proved example to construct two new rigid Calabi-Yau threefolds
that realize Hecke eigenforms of weight $4$
(one provably and one conjecturally).
\end{abstract}

\maketitle

\section{Introduction}\label{sec:intro}
The bijective correspondence between rational Hecke eigenforms of weight $2$ 
of minimal level $N$
and elliptic curves over $\Q$ of conductor $N$ up to isogeny is one of the
central concepts in modern number theory.  It has been generalized in many 
ways, for example by replacing elliptic curves over $\Q$ with elliptic curves
over a totally real number field: the corresponding objects then appear to
be Hilbert modular forms.  In another direction, we might seek to increase the
weight of the form and the dimension of the corresponding variety.  

For odd $k$, the presence of $-1$ in $\Gamma_0(N)$ forces rational
eigenforms to have an imaginary quadratic Nebentypus character of conductor
dividing $N$ and to be of CM type.
Sch\"utt showed \cite{schutt-cm}
that rational newforms of CM type and weight $k$ up to twist
are in bijection with imaginary quadratic fields whose class group is killed
by $k-1$.  For $k=2$ this is essentially the classical fact that elliptic curves
with complex multiplication up to twist correspond to imaginary quadratic fields
of class number $1$; for $k=3$ such fields have been studied for centuries in
the guise of Euler's numeri idonei, which are the positive integers $n$
for which the primes represented by a quadratic form of discriminant $-n$ are
determined by congruence conditions.  There is a list of such integers,
found in \cite{elkies-schutt}, which is
known to be complete if the generalized Riemann hypothesis is assumed; without
this assumption, the list is complete with at most one exception.

Elkies and Sch\"utt \cite{elkies-schutt}
then showed that, in the case $k=3$, every 
modular form $f$ corresponding to a known imaginary quadratic field with class
group killed by $2$ is realized in a singular K3 surface $S_f$, 
in the sense that the
action of $\Gal(\bar \Q/\Q)$ on its \'etale cohomology includes the 
representation associated to $f$.  That is, the semisimplification of
the representation
$\phi: \Gal(\bar \Q/\Q) \to GL_{22,\Z_\ell}$ on $H^2_{\etale}(S_f,\Z_\ell)$ 
has a $2$-dimensional component $\rho_\ell$ such that
$\tr \rho_\ell(\Frob_p) = a_p$ for all but finitely many $p$, where $a_p$ is the
eigenvalue of $S_f$ for the Hecke operator $T_p$.
The complement of this component is the twist of a representation of
finite image giving the Galois action on the Picard group of $S_f$ 
by the cyclotomic character.
More concretely, the number of $\F_p$-points of $S_f$ is equal to 
$p^2+1+a_p+pn(p)$, where $n(p) \in \Z$ depends only on the decomposition of $p$
in the field of definition of the Picard group, which is a finite extension
of $\Q$.

\begin{defn}\label{realize}
Let $f$ be a newform with integral coefficients and $V$ a smooth variety
over $\Q$.  If for
all $\ell$ the $\ell$-adic representation attached to $f$ is a component of
the $\ell$-adic representation attached to $H^i_{\etale}(V,\Z_\ell)$, we say that 
$V$ {\em realizes} $f$.
\end{defn}

Thus the result of Elkies and Sch\" utt shows that every integral newform of 
weight $3$ is realized by a K3 surface with one possible exception.  It is
a very important question, asked by Mazur, van Straten, and others, whether
an analogous result holds for weight $>3$.
To discuss this problem, we must introduce Calabi-Yau varieties.

\begin{defn}[{\cite[Section 1.1]{meyer-diss}}]\label{def-smooth-cy}
A {\em Calabi-Yau variety} of dimension $d$
is a smooth proper variety with trivial canonical class such
that $h^{i,0} = 0$ for all $0 < i < d$.
\end{defn}

\begin{defn}\label{def-cy} A {\em singular Calabi-Yau variety}
is a variety $V$ that is birationally equivalent to a smooth Calabi-Yau variety
and is a limit of smooth Calabi-Yau varieties.
\end{defn}

\begin{example} Let $n > 1$ and let $c_1, \dots, c_{n-d}$ be positive integers
adding to $n+1$.  Then a sufficiently general intersection of hypersurfaces
of degree $c_1, \dots, c_{n-d}$ in $\P^n$ is a Calabi-Yau variety of dimension 
$d$.  If $X$ is such a complete intersection such that the canonical bundle on
the nonsingular subset of such an $X$
extends to all of $X$ and $X$ admits a resolution of singularities with
trivial canonical bundle, it is a singular Calabi-Yau variety.
\end{example}
%%% ref for du Val sings? or canonical ones?

\begin{defn}[{\cite[Section 1.1]{meyer-diss}}] A smooth
Calabi-Yau variety of dimension $d$ for which $h^{d-1,1} = 0$ is
said to be {\em rigid}.  A singular Calabi-Yau variety is
rigid if it has a resolution of singularities which is a rigid Calabi-Yau
variety.
\end{defn}

A considerable amount of work has been devoted to finding
rigid Calabi-Yau threefolds defined over $\Q$.  
It is known \cite{gouvea-yui}, \cite{dieu-man} that the 
representation on $H^3$ of such a variety has the same semisimplification
as the representation associated to an integral Hecke eigenform of weight $4$.
Meyer \cite{meyer-diss} gives
a large number of examples of rigid Calabi-Yau threefolds and corresponding
modular forms, but it seems that the only general
method of constructing a threefold starting from a modular form of weight $4$
(as opposed to
recognizing a modular form in a variety constructed without reference to it)
is that of \cite{cynk-schutt}, which is applicable only to eigenforms of
CM type coming from imaginary quadratic fields of class number $1$ or $3$.

It is not known whether there are finitely or infinitely many
deformation families of Calabi-Yau threefolds.  %% ref? 
On the other hand, tables such as that of \cite[Appendix~C]{meyer-diss}
make it plausible that there are infinitely many 
newforms of weight $4$ with integer coefficients up to twist, although this
may be quite difficult to prove (before the famous work of Wiles and
Taylor-Wiles the analogous
statement was apparently not known for weight $2$).
In any case, each family can contribute
realizations of only finitely many newforms up to twist on rigid Calabi-Yaus,
so if there are finitely many families and infinitely many newforms
up to twist it is not possible for all of them to be realized.

%% Even if a Calabi-Yau threefold is not rigid, understanding the Galois action
%% on its \'etale $H^3$ is still very interesting.  In some cases 
%% the representation contains $H^{3,0} \oplus H^{0,3}$ as an irreducible component,
%% and then there is a congruence $N_p \equiv a_p+1$, where $N_p$ is the number
%% of points over $\F_p$ and $a_p$ is the eigenvalue of $T_p$ as before.
%% In \cite{meyer-diss} the theory of this situation is treated in section 1.5.2,
%% with examples passim.  In other cases, as in \cite{cs-hmf},
%% the representation is related to Hilbert modular forms
%% of weight $(2,4)$.  In general one expects the representations to be related
%% to some kind of automorphic form, although there may be no independent method
%% of finding a basis for the appropriate space.

The arithmetical results of this paper are summarized in the following theorem:
\begin{thm}\label{main} There are rigid Calabi-Yau threefolds defined over $\Q$
for which the representation on $H^3$ is the same as that associated to the
modular forms labeled $13/1$, $78/4$ in \cite[Appendix C]{meyer-diss}.  If
the tables of \cite{tables} are complete for cubic extensions of $\Q$
unramified outside $\{2,3,5,13,23\}$, the same holds for the form $390/5$.
\end{thm}
For the form of level $390$, there was a known nonrigid Calabi-Yau threefold
whose $H^3$ was conjectured to contain this representation; for the other
two, even this was not available.
The construction of the threefold realizing the newform of level $13$
was inspired by mathematical work on the foundations of quantum field theory.
\begin{defn}\label{graph-var}
Let $G$ be a graph with vertex set $V = \{v_1,\dots,v_m\}$ and edge set
$E = \{e_1,\dots,e_n\}$, and let $R = \Z[x_1,\dots,x_n]$.  Let $T$ be the
set of spanning trees of $G$.  As in
\cite[Section 1.1]{k3-phi4}, define
$$\label{psig}\begin{aligned}
\Psi_G &= \sum_T \prod_{e_i \notin T} x_i.\end{aligned}$$
Let $X_G$ be the affine variety $\Psi_G = 0$; it is called the
{\em graph hypersurface} of $G$.  If $m \ge 3$ then the
number of $\F_p$-points of $X_G$ is a multiple of $p^2$, so we define
$c_2(p) = c_2(G)_p$ to be this number divided by $p^2$ and reduced mod $p$.
\end{defn}

If no restriction is placed on $G$, then $X_G$ contains 
essentially arbitrary motives \cite{belkale-brosnan}.
We now suppose that $G$ is obtained by deleting a single vertex from a 
$4$-regular graph, which places us in the setting known as $\phi^4$-theory. 
In this context it had been suggested % by Kontsevich
that the number of 
$\F_p$-points might be a polynomial in $p$.
Brown and Schnetz refuted this by giving an example \cite[Section 6.2]{k3-phi4} 
in which $c_2(p) \equiv a_p \bmod p$ where the $a_p$ are the eigenvalues
for a Hecke eigenform of weight $3$ and level $7$.  To do so, they related
$c_2(p)$ to the number of points on a K3 surface of Picard number $20$.
In later work \cite{mf-qft},
they found $16$ Hecke eigenforms that appear to match
the $c_2(p)$ for various graphs.

For the five graphs of \cite[Figure 5]{mf-qft} that
numerically match a modular form of weight $4$, we construct a threefold
for which the number of $\F_p$-points is congruent to $c_2(p) \bmod p$.
For the form of level $13$, we will show that it is a rigid Calabi-Yau threefold
and that the number of $\F_p$-points is indeed congruent mod $p$ to the 
eigenvalue of the newform of weight $4$ and level $13$.  
In addition, by studying a family of threefolds that specializes to this one,
we construct two new rigid Calabi-Yau threefolds,
one realizing the modular form of level $78$, the other apparently that of
level $390$.

For each of the other four graphs, matching the forms of levels $5, 6, 7, 17$, 
we conjecture that the variety is a rigid Calabi-Yau threefold and give a
formula that appears to count its $\F_p$-points.
These forms are already
known to be realized on rigid Calabi-Yau threefolds, but the construction here 
is completely different.  In
light of the Tate conjecture \cite[Conjecture 1.10]{meyer-diss}
we expect ours to be in correspondence with the previous ones.  However, we
do not know how to exhibit such a correspondence.
In addition, we will verify that the graphs labeled $(3,8)$ and $(3,12)$ in
\cite[Figure 5]{mf-qft} correspond to singular K3 surfaces realizing
modular forms of the indicated level.  
%% In future work I hope
%% to use these methods to explain some of the $c_2(p)$ sequences in \cite{mf-qft}
%% that have no known description in terms of modular forms.

It would not have been possible to write this paper without substantial 
electronic computation.  The Magma system \cite{magma}
was used.

{\em Acknowledgments.}  I would like to thank Karen Yeats for giving a
very inspiring seminar that sparked my interest in this field of 
research and for encouraging me to pursue this project.  I also thank
Henri Darmon and Bas Edixhoven for their explanations of the Kuga-Sato
construction,
Colin Ingalls and Owen Patashnick for their careful readings of an earlier
draft of this paper and for many helpful conversations and references, and
Noriko Yui for a valuable discussion of the contents of the paper and the
general area of modular Calabi-Yau threefolds.

\section{The five-invariant}
In the introduction we associated a polynomial to a graph $G$ and defined
the graph hypersurface associated to it. % in (\ref{psig}).
Unfortunately, the variety is difficult to
work with because it is of large dimension.  In this section we describe some
of the constructions used by Brown, Schnetz, and Yeats
\cite{k3-phi4}, \cite{mf-qft}, \cite{bsy}
to obtain varieties of lower dimension for which the point count over $\F_p$
is the same mod $p$.
In contrast to their work, we will regard our varieties
as projective rather than affine.

First, we choose an arbitrary orientation of the edges of $G$.  We then 
define a symmetric $(n+m) \times (n+m)$ matrix $\tilde M_G$
over $\Z[x_1,\dots,x_n]$ (recall that $m, n$ are the
numbers of vertices and edges of $G$) as in \cite[Definition 8]{k3-phi4}.
That is, the upper left block is a diagonal matrix with entries $x_i$,
the upper right and lower left blocks are the oriented incidence matrix and
its transpose, and the lower right block is $0$.

The sum of the last $m$ rows of this matrix is $0$, so we choose one 
arbitrarily and delete it as well as the corresponding column to obtain
$M_G$.  The determinant of $M_G$ is $\pm \Psi_G$.

\begin{defn}[{\cite[Definition 9]{k3-phi4}}]
Let $I, J, K$ be subsets of the set of
edges with $|I| = |J|$.  Let $M_G(I,J)_K$ be the matrix $M_G$ with the rows
corresponding to $I$ and columns corresponding to $J$ removed, and the
variables associated to $K$ set to $0$.  Let $\Psi^{I,J}_{G,K}$ be its
determinant.
\end{defn}

\begin{defn}[{\cite[Definition 12]{k3-phi4}}] Let $i,j,k,l,m$ be five distinct
edges in $G$.  The {\em five-invariant} ${}^5\Psi_G(i,j,k,l,m)$ is 
defined to be $$\det 
\begin{pmatrix}\Psi^{ij,kl}_{G,m} & \Psi^{ik,jl}_{G,m}\cr
\Psi^{ijm,klm}_{G,\emptyset} & \Psi^{ikm,jlm}_{G,\emptyset}
\end{pmatrix}.$$
\end{defn}

\begin{remark} Let $G$ be the deletion of a vertex from a $4$-regular graph,
so that $G$ has $n-1$ vertices and $2n-4$ edges.  Then the spanning trees 
have order $n-2$, which means that $\Psi_G$ is homogeneous of degree $n-2$ in
$2n-4$ variables.  On the other hand, $\Psi^{I,J}_{G,K}$ has degree
$\deg \Psi_G - |I|$.  Thus in general the degree of ${}^5\Psi_G(i,j,k,l,m)$
is $n-4+n-5 = 2n-9$.  Since it does not depend on the variables associated to
edges $i,j,k,l,m$ it is a polynomial in $2n-9$ variables.  Observe also that
$\Psi_G$ has degree at most $1$ in every variable, so ${}^5\Psi_G(i,j,k,l,m)$
has degree at most $2$.
\end{remark}

Up to sign, this determinant depends only on the set $\{i,j,k,l,m\}$,
not on the order.  Brown and Schnetz prove the following:

\begin{thm}[{\cite[Theorem 3]{k3-phi4}}]
Let $G$ be a graph with at least five edges $e_1, \dots, e_5$
such that $N_G = 2h_G$ and $N_\gamma > 2h_\gamma$ for all strict subgraphs
$\gamma \subsetneq G$, where $N$ is the number of edges of a graph and
$h$ its first Betti number.
Then for all primes $p$, the number of solutions to the equation
${}^5\Psi_G(e_1,e_2,e_3,e_4,e_5) = 0$ over $\F_p$ is congruent mod $p$ to
$-c_2(G)_p$.
\end{thm}

\section{Reduction}\label{sec:red}
In this section we describe the two main constructions used in, among
others, \cite{k3-phi4}, \cite{mf-qft}, \cite{bsy},
that allow $c_2(G)_p$ to be computed from a variety
of manageable dimension.  First we restate them in terms of projective
rather than affine varieties.  Then we explain how
these reductions can be interpreted in terms of blowups of projective
varieties along linear subspaces and generalize to obtain 
a reduction technique not previously used in this context.

\begin{defn} Throughout the paper, the coordinates in an ordinary projective 
space $\P^n$ will always
be denoted $x_0, \dots, x_n$.  In weighted projective space $\P(k,1,1,\dots,1)$
the coordinates will be $t, x_0, \dots, x_n$, and in product projective space
$\P^m \times \P^n$ they will be $x_0, \dots, x_m, y_0, \dots, y_n$.
\end{defn}

\begin{defn} If $f_1, \dots, f_n$ are homogeneous polynomials and $q$ is
a prime power, let
$[f_1,\dots,f_n]_q$ be the number of points on the projective variety defined
by $f_1=\dots=f_n=0$ over $\F_q$.  If $V$ is a variety, let $[V]_q$ be the
number of $\F_q$-points of $V$.
\end{defn}

\begin{remark}
Let $V_P, V_A$ be the projective and affine varieties over $\F_q$ defined by
$f_1=\ldots=f_n=0$, where the $f_i$ are homogeneous polynomials.  Then
$[V_A]_q = (q-1)[V_P]_q+1$, so the sequences $[V_A]_q, [V_P]_q$ contain the
same information.
%However, we will find projective constructions such as the blowup useful later.
\end{remark}

\begin{defn} Let $V, W$ be projective varieties over $\Q$.
We say that $V$ and $W$ are 
{\em prime-similar} if $[V]_p-1 \equiv (-1)^{\dim W - \dim V}([W]_p-1) \bmod p$
for all but finitely many primes $p$.  If $W$ is a point we say that $V$ is 
{\em prime-trivial}.
\end{defn}

\begin{example} The Chevalley-Warning theorem \cite[Theorem 1.3]{serre-ca}
states that a subvariety of
$\P^n$ defined by equations the sum of whose degrees is at most $n$ is
prime-trivial.
\end{example}

\begin{lemma}\label{triv-blowup} 
Let $V$ be a variety over $\Q$ smooth along the subvariety $W$, and let
$\tilde V$ be $V$ blown up along $W$.  Then $\tilde V$ is prime-similar to 
$V$.
\end{lemma}

\begin{proof}
Let $p$ be a prime for which $V/\F_p$ is smooth along $W/\F_p$ (this excludes
only finitely many primes).  Then
the $\F_p$-points of $V \setminus W$ are in 
bijection with points of $\tilde V \setminus E$, while $E$ maps to $W$ with
fibres that are projective spaces of dimension $c-1$ over $\F_p$, where $c$ is 
the codimension of $W$ in $V$.  Thus $\tilde V$ has 
$(p^{d-1}+p^{d-2}+\dots+p)[W]_p$ more $\F_p$-points than $V$.
\end{proof}

%% More generally, if
%% $V_1 \to V_2$ is surjective on points and the fibres are prime-trivial, 
%% then $V_1$ and $V_2$ are prime-similar.

\begin{example}[{\cite[Lemma 16]{k3-phi4}}]
Let $V$ be a hypersurface in $\P^n$ ($n > 1$) defined by a polynomial
$fx_0 + g$ of degree $1$ in $x_0$.  Then $V$ is prime-similar to the
subvariety $W$ of $\P^{n-1}$ defined by $f$.  To see this, note first that
$P_0 = (1:0:\ldots:0) \in V$ (here we use that $n > 1$) and consider the 
rational map from $\P^n$ to $\P^{n-1}$ given by projecting away from 
$P_0$.  Above every point of $\P^{n-1} \setminus W$ there is one
point of $V$.  Above every point of $W$ the number of points of $V$ not equal
to $P_0$ is $p$ if $g = 0$ there and otherwise $0$.  Thus, if $W$ has $k$ 
points, then the number of points of $V$ is congruent mod $p$ to 
$$(p^{n-1}+p^{n-2}+\dots+1-k)+1 \equiv 2-k \pmod p,$$ as desired.

We refer to this as {\em linear reduction}.
\end{example}

\begin{example}[{\cite[Lemma 27]{k3-phi4}}]
Let $V$ be a hypersurface in $\P^n$ defined by $fg = 0$, where
$f, g$ are both of degree $1$ in the variable $x_0$ and $fg$ is of total degree
$n+1$.  Then $V$ is prime-similar to the subvariety of $W$ defined by
$[f,g]_{x_0}$, the resultant of $f, g$ with respect to $x_0$.  We call this
{\em resultant reduction}.
\end{example}

\begin{remark} Both of these reductions preserve the property of the total
degree of a polynomial being equal to the number of variables, as well as that
of the degree in any one variable being at most $2$.
\end{remark}

Now we will reinterpret linear reduction in
terms of a standard geometric construction: this will show us how to
generalize it.
The variety $fx_0+g = 0$ in $\P^n$ is singular
at $(1:0:\ldots:0)$.  If we blow up this point, the exceptional divisor is
defined by $f = 0$.  On the other hand, the fibres of the projection to 
$\P^{n-1}$ are points or lines, so the blowup is prime-trivial and
the original variety is prime-similar to the exceptional divisor.
More generally, we consider varieties singular along a coordinate plane
rather than at a point.

%% The normal reduction was defined in these terms in the first place, by blowing
%% up the linear subspace of codimension $2$ to obtain a prime-trivial variety
%% and replacing the original variety by the exceptional divisor.
%% Another example is considered in section \ref{w3-lev8} below.  We summarize 
%% these results in a general method of reduction.

\begin{defn}
Let $V$ be a hypersurface of degree $n+1$ in $\P^n$, defined by $f = 0$.
Suppose that, for some $k$ with $2 \le k \le n$ and
$0 \le i_1 < i_2 < \dots < i_k \le n$,
we have $f \in (x_{i_1},x_{i_2}, \dots, x_{i_k})^k$.
We then define the {\em subspace reduction} $L(V,k) = L(V,\{i_1,\dots,i_k\})$
of $V$ by the $x_{i_j}$ as the exceptional divisor in the blowup of
$V$ along $x_{i_1} = \ldots = x_{i_k} = 0$.
It is naturally a hypersurface of bidegree
$(n-k+1,k)$ in $\P^{n-k} \times \P^{k-1}$.
\end{defn}

%% Of course the number of $\F_p$-points is not changed by permuting
%% the variables, so we may apply subspace reduction to any set of variables, as
%% long as every monomial satisfies the degree condition.  The restriction to
%% $x_0,\ldots,x_{k-1}$ is made only to simplify the notation.

%% \begin{remark}
%% The general fibre of $L(V,k)$ over $\P^{k-1}, \P^{n-k}$ is of
%% degree $n-k+1, k$ respectively.
%% \end{remark}

\begin{prop}\label{lin-sim}
The subspace reduction of $V$ is prime-similar to $V$.
\end{prop}

\begin{proof}
Let $\tilde V$ be the blowup.  Since 
$[\tilde V]_p = [V_p] - [x_0,\ldots,x_{k-1}]_p + [L(V,k)]_p$, it suffices to prove
that $\tilde V$ is prime-trivial.  To do this, we will show that
every fibre of the projection $\pi: \tilde V \to \P^{k-1}$ is prime-trivial.
Since $\P^{k-1}$ is prime-trivial, this implies that $[\tilde V]_p$
is a sum of $c$ 
integers congruent to $1 \bmod p$, where $c \equiv 1 \bmod p$.  It follows that
$[\tilde V]_p \equiv 1 \bmod p$, i.e., that $\tilde V$ is prime-trivial.

Let $\rho$ be the rational map $\P^n \to \P^{k-1}$ defined by
$(x_0:x_1:\ldots:x_{k-1})$.
Let $P = (c_0:c_1:\dots:c_{k-1})$ 
be a point of $\P^{k-1}$.  Then $\pi^{-1}(P)$ (viewed naturally
inside $\P^n$ identified with $\P^n \times P$) is equal to $\rho^{-1}(P) \cap V$
with the component supported on $x_0 = \ldots = x_{k-1}$ removed.  There are
two possibilities.

\begin{enumerate}
\item If $\rho^{-1}(P) \cap V = \rho^{-1}(P)$, then it is a linear subspace of
codimension $k-1$, which is certainly prime-trivial.
\item If not, consider the subscheme of $V$ defined by the $k-1$ linear
equations that define $P$.  If $c_i \ne 0$, then on this subscheme all of the
$x_0, \dots, x_{k-1}$ are constant multiples of $x_i$.  When we substitute for
them in the polynomial defining $V$, the result is a multiple of $x_i^k$.
This corresponds to the component supported on $x_0 = \ldots = x_{k-1}$,
so we divide out the highest power of $x_i^k$ to obtain a polynomial in
$x_k,\dots,x_n$ defining the fibre of $\pi$.  This polynomial in $n-k+1$
variables has degree at most $n-k$, so the hypothesis of the Chevalley-Warning
theorem holds and the fibre is prime-trivial.
\end{enumerate}
\end{proof}

\begin{remark}
This proposition is our justification for considering graph hypersurfaces
and their reductions as projective rather than
affine varieties.  The blowup of a projective variety is projective but that
of an affine variety is not affine, so this construction is more natural for
projective varieties.

Note also that this proposition does not contradict Lemma \ref{triv-blowup},
because $V$ is singular along $x_0 = \ldots = x_{k-1} = 0$.
\end{remark}

We now describe the case $k = 2$ in more detail.
Let $V$ be defined by $f = 0$ with $f \in (x_0,x_1)^2$.  Then $L(V,2)$
is a prime-similar variety in $\P^{n-2} \times \P^1$.
Define $c_0, c_1, c_2$ to be polynomials that satisfy 
$f = c_0 x_0^2 + c_1 x_0 x_1 + c_2 x_1^2$ (any such 
representation will do, but for uniqueness one may choose $c_0$ respectively
$c_2$ not to depend on $x_1$ resp.{} $x_0$).
The fibre above $P \in V(f)$ is a single point if $x_0, x_1$ are not both $0$;
otherwise it consists of all projective points $(y_0:y_1)$ satisfying
$c_0 y_0^2 + c_1 y_0 y_1 + c_2 y_1^2 = 0$.

\begin{prop}\label{to-hyp}
Suppose that $c_1^2 - 4c_0 c_2 = gh$, where $\deg g = n-2$.
Then $L(V,2)$ is birationally equivalent and prime-similar
to a hypersurface $H_n$ of degree $n$ in $\P^{n-1}$.
\end{prop}

\begin{proof} 
Let $W = \P(n-1,1,1\dots,1)$ be a weighted projective space of dimension $n-1$.
Let the $c_i'$ be obtained from the $c_i$ by evaluating them at 
$(0,0,x_0,\dots,x_{n-2})$, where the $x_i$ are now the weight-1 coordinates of 
$W$.  Observe that the degree of the $c_i'$ is $n-1$, by our hypotheses on the 
monomials making up $P$.  We show that $L(V,2)$ is birationally
equivalent and prime-similar to the subvariety $V_N$ of $W$ defined by 
$t^2 + c_1' t + c_0' c_2' = 0$.  

Indeed, let us fix a point in $\P^{n-2}$ and consider the points above it in
each variety.  In $V_N$, they correspond to roots of $t^2 + c_1 t + c_0 c_2$;
in $L(V,2)$, to points satisfying
$c_0 y_0^2 + c_1 y_0 y_1 + c_2 y_1^2$.  These are in bijection unless 
$c_0 = c_1 = c_2$, in which case there is $1$ of the former and $p+1$ of the 
latter, which numbers are congruent mod $p$.  Finally, the point
$(1:0:0:\ldots:0)$ is not on $V_N$.

To go from $V_N$ to the desired hypersurface,
we complete the square in $t$, which can only affect the
number of $\F_p$-points if $p=2$.  This gives us $t^2 - gh = 0$ in $W$, which we
will still call $V_N$.  Then $H_n$ is defined by $v_0^2g - h = 0$.  There is
an obvious invertible rational map $H_n \cdots \to V_N$
taking $(v_0:\ldots:v_{n-1})$ to
$(v_0g:v_1:\ldots:v_{n-1})$.  This map is defined everywhere except at 
$(1:0:\ldots:0)$, which is blown up to $v_0 = g = 0$.
The inverse map can be given by two sets of equations
$$(v_0:gv_1:\ldots:gv_{n-1}),\quad (h:v_0v_1:\ldots:v_0v_{n-1})$$ and has base 
locus $v_0 = g = h = 0$.  So the points of $H_n$ consist of $(1:0:\ldots:0)$,
one point for every point of $V_N$ outside $v_0 = g = 0$, and either 
$0$ or $p$ points for every point of $v_0 = g = 0$.
Applying the Chevalley-Warning theorem to $v_0 = g = 0$ we see that the 
numbers of points on the two varieties are congruent mod $p$.
\end{proof}

\begin{remark} This construction along with several related ones is given 
in \cite[Section 1.7.2]{meyer-diss} for double covers of $\P^3$ branched along
the union of eight planes, but the general case was known earlier.
The proposition above is often, though not always,
applicable to the subspace reduction along a line of a variety
arising from a graph polynomial.
%% ; unfortunately the reason for this is not clear.
\end{remark}

\begin{defn} When $k = 2$ we refer to the subspace reduction as the
{\em normal reduction}, because the blowup can be interpreted as the 
normalization of $V$ along the divisor $x_0 = x_1 = 0$.
\end{defn}

Similarly, in the case $k = n-1$, we obtain a hypersurface
in $\P^1 \times \P^{n-2}$ that can be viewed as a double cover of $\P^{n-2}$,
and if the branch locus is suitably reducible we can revert to a hypersurface
of degree $n$ in $\P^{n-1}$: for some examples of this see sections
\ref{w3-lev8}, \ref{w4-lev5}, \ref{w4-lev6}.  If neither $k-1$ nor $n-k$ is
$0$ or $1$, it is not clear how to continue the reduction process.
%, but that case does not arise in this paper.

%% There are some examples where the branch locus is not suitable.  For example,
%% the hypersurface of the graph labeled $i_{53}$ in Figure 6 of \cite{mf-qft} 
%% can be reduced to a double cover of $\P^4$ branched along a union of divisors
%% of degree $1,2,7$ or $2,8$, but Proposition \ref{to-hyp} does not apply to 
%% these varieties.
%% \end{remark}

The resultant reduction cannot immediately be described in terms of the subspace
reduction, unless one of the factors is linear, but it is certainly related.
Consider the variety
in $\P^n$ defined by $(f_1x_0 + f_0)(g_1x_0+g_0) = 0$, where $f_i, g_i$ do
not depend on $x$ and $f_1, g_1$ have no common factor.  
It is prime-similar to the
variety defined by $f_1x_0 + f_0 = g_1x_0 + g_0$ (even this can be seen by
blowing up along the intersection).  On this variety, we
again blow up $(1:0:\ldots:0)$, obtaining an exceptional divisor defined by
$f_1 = g_1 = 0$.  If $\deg(f_1x_0 + f_0)(g_1x_0+g_0) = n+1$, then 
$\deg f_1 + \deg g_1 = n-1$, so the exceptional divisor is subject to the
Chevalley-Warning theorem.  The projection from the blowup to $\P^{n-1}$
again has linear fibres and its image is defined by the elimination ideal 
for $x_1,\dots,x_n$
of the ideal $(f_1x_0 + f_0,g_1x_0 + g_0)$, which (since these have no
common factor) is generated by the resultant $f_1g_0 - f_0g_1$.

\section{The modular form of weight $4$ and level $13$}\label{sec:4-13}
In this section we will describe a rigid Calabi-Yau threefold for which the
representation on $H^3$ gives the unique newform of level $4$ and weight
$13$ with rational eigenvalues.
By considering special elements of families containing this threefold
we will also realize newforms of level $78$ and, apparently, $390$.

We start from the graph labeled $(4,13)$ in \cite[Figure 5]{mf-qft}.
This is $\#243$ in the list of $4$-regular graphs on $11$
vertices produced by {\tt genreg} \cite{genreg}, whose numbering of the vertices
we use in place of that of \cite{mf-qft}.  By applying the reductions of
section \ref{sec:red} to its graph polynomial 
we will produce the desired threefold.

\begin{myempty}\label{wt-4-13}
\end{myempty}
\begin{figure}[ht]
\begin{center}
\begin{tikzpicture}
\coordinate[label=right:$1$] (v1) at (0:1); % {$1$};
\coordinate[label=right:$2$] (v2) at (2*pi/11 r:1); % {$2$};
\coordinate[label=above right:$3$] (v3) at (4*pi/11 r:1); % {$3$};
\coordinate[label=above:$4$] (v4) at (6*pi/11 r:1); % {$4$};
\coordinate[label=above left:$5$] (v5) at (8*pi/11 r:1); % {$5$};
\coordinate[label=left:$6$] (v6) at (10*pi/11 r:1); % {$6$};
\coordinate[label=left:$7$] (v7) at (12*pi/11 r:1); % {$7$};
\coordinate[label=below left:$8$] (v8) at (14*pi/11 r:1); % {$8$};
\coordinate[label=below:$9$] (v9) at (16*pi/11 r:1); % {$9$};
\coordinate[label=below right:$10$] (v10) at (18*pi/11 r:1); % {$10$};
\coordinate[label=right:$11$] (v11) at (20*pi/11 r:1); % {$11$};
\draw (v1) -- (v3);
\draw (v1) -- (v4);
\draw (v1) -- (v5);
\draw (v2) -- (v3);
\draw (v2) -- (v6);
\draw (v2) -- (v7);
\draw (v3) -- (v8);
\draw (v3) -- (v9);
\draw (v4) -- (v5);
\draw (v4) -- (v6);
\draw (v4) -- (v8);
\draw (v5) -- (v10);
\draw (v5) -- (v11);
\draw (v6) -- (v7);
\draw (v6) -- (v10);
\draw (v7) -- (v9);
\draw (v7) -- (v11);
\draw (v8) -- (v10);
\draw (v8) -- (v11);
\draw (v9) -- (v10);
\draw (v9) -- (v11);
\fill (v1) circle (2pt) (v2) circle (2pt) (v3) circle (2pt) (v4) circle (2pt) (v5) circle (2pt) (v6) circle (2pt) (v7) circle (2pt) (v8) circle (2pt) (v9) circle (2pt) (v10) circle (2pt) (v11) circle (2pt);
%% \node[below=1cm,align=flush center,text width=10cm] at (v9)
%% {{\sc Figure 1.}  A graph for which the $c_p$ are congruent mod $p$ to the
%% eigenvalues of the newform of weight $4$ and level $13$.};
\end{tikzpicture}
\end{center}
\caption{A graph for which the $c_p$ are congruent mod $p$ to the
eigenvalues of the newform of weight $4$ and level $13$.}
\end{figure}
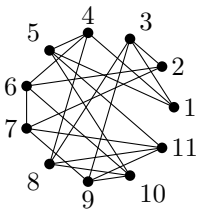

We first delete vertex $1$, then take the $5$-invariant
for the edges $$(2,3),(2,6),(2,7),(3,9),(6,7),$$
and then we apply linear or resultant reduction to the edges
$$(3,8),(4,6),(5,10),(4,5),(4,8),(5,11),(9,11)$$ in that order.
This done, we have a polynomial of degree $6$ in $6$ variables to which we can
apply normal reduction, since it defines a subvariety of $\P^5$ singular
along the $3$-plane $x_1 = x_2 = 0$.
Next we use Proposition \ref{to-hyp} to obtain a quintic $Q_1$ in $\P^4$.
It is defined by the polynomial
$$
\begin{aligned}
& -x_0^2x_1x_2x_3/4 + x_1x_2^3x_3 - x_0^2x_1x_3^2/4 + x_1x_2^2x_3^2 -
   4x_1^2x_2^2x_4 - 4x_1x_2^3x_4 \\
&\qquad - x_0^2x_1x_3x_4/4 - x_0^2x_2x_3x_4/4 - 4x_1^2x_2x_3x_4 - 3x_1x_2^2x_3x_4 + 
   x_2^3x_3x_4 \\
&\qquad - x_0^2x_3^2x_4/4 +  x_2^2x_3^2x_4 - 4x_1^2x_2x_4^2 - 4x_1x_2^2x_4^2 - 4x_1^2x_3x_4^2 - 4x_1x_2x_3x_4^2\\
\end{aligned}
$$
and we know that $[Q_1]_p \equiv [H]_p/p^2 \bmod p$, where $H$ is the
hypersurface defined by the graph polynomial of $G$.
%% Using a computer we immediately confirm that $[Q_1]_p + a_p \equiv 1 \bmod p$
%% for odd primes $< 100$, where the $a_p$ are the coefficients of the unique
%% newform of weight $4$ and level $13$.

\begin{remark}
We expect $Q_1$ to be birational to a rigid Calabi-Yau threefold,
So, in view of \cite[Corollary 1.14]{reid}, it is unsurprising
that the codimension-$2$ components
of the singular locus are lines of compound du~Val (cdV) singularities.
Of these, $10$ 
are of type $A_1$ and $2$ of type $A_3$.  Some of the codimension-$3$ 
components are embedded points where these specialize to compound $A_2$ and 
$A_4$ singularities.  However, the singularities where the components meet are 
not easy to deal with directly.
\end{remark}

The most difficult singular points are those in the support of components
of the singular subscheme whose multiplicity is large.
We would like to blow these up; however, it is impractical to work in the
product of projective spaces that would result.
To improve our model, we will use a map defined by the linear system of
quadrics vanishing along the reduced subschemes of certain components of the
singular subscheme.
Namely, consider the map $\P^4 \to \P^4$
defined by $$(x_i) \to (x_0(x_2+x_3):x_1x_2:x_1x_3:x_1x_4:x_2(x_2+x_3)).$$
%% (This was found by considering linear systems of quadrics that vanish at
%% the worst of the singular points and at several others in such a way as to
%% define a map to $\P^4$ with quintic image.)  
A computation in Magma tells us
that the map has degree $1$ on $Q_1$ and that the image $Q_2$ is defined by
$$\begin{aligned}
& x_0^2x_1^2x_2 + x_0^2x_1x_2^2 + x_0^2x_1x_2x_3 + 16x_1^3x_3x_4 + 
    x_0^2x_2x_3x_4 + 32x_1^2x_2x_3x_4 \cr
&\quad + 16x_1x_2^2x_3x_4 + 16x_1^2x_3^2x_4 + 32x_1x_2x_3^2x_4 + 16x_2^2x_3^2x_4 - 
    4x_1^2x_2x_4^2 - 4x_1x_2^2x_4^2 \cr
&\quad + 16x_1^2x_3x_4^2  + 12x_1x_2x_3x_4^2 + 16x_1x_3^2x_4^2 + 16x_2x_3^2x_4^2 - 
    4x_2x_3x_4^3.\end{aligned}$$
The discriminant of the equation defining $Q_1$ with respect to $x_1$
factors into polynomials of degrees $5,2,1$.  For $Q_2$, on the other hand,
the factors have degrees $4,2,1,1$, indicating that we have made progress.

Continuing in the same way, we map $\P^4 \to \P^4$ by
$$(x_i) \to ((x_0-2x_4)(x_1+x_2):x_1(x_1+x_2):x_1x_3:x_2x_3:x_3x_4).$$
Again the map is of degree $1$ on $Q_2$, and the image $Q_3$ is defined by
$$\begin{aligned}
&x_0^2x_1x_2x_3 + x_0^2x_2^2x_3 + 16x_1^2x_2^2x_4 + 16x_1x_2^3x_4 + 
    4x_0x_1^2x_3x_4 \\
&\quad + x_0^2x_2x_3x_4 + 4x_0x_1x_2x_3x_4 + 
    32x_1^2x_2x_3x_4 + 48x_1x_2^2x_3x_4 + 16x_1^2x_3^2x_4 \\
&\quad +  48x_1x_2x_3^2x_4 + 16x_1x_3^3x_4 + 16x_1^2x_2x_4^2 + 
    16x_1x_2^2x_4^2 + 4x_0x_1x_3x_4^2 \\
&\quad +  16x_1^2x_3x_4^2 + 
    32x_1x_2x_3x_4^2 + 16x_1x_3^2x_4^2,\\
\end{aligned}
$$
whose discriminant $D_0$ with respect to $x_0$ factors into polynomials
of degrees $4,1,1,1,1$.  Using Proposition \ref{to-hyp} in reverse, we find
that the hypersurface $H_3$ in $\P(4,1,1,1,1)$ defined by $t^2 = D_0$ is
birationally equivalent to $Q_3$.
Let $S_1, \dots, S_5$ be the components of the branch locus of the map
$H_3 \to \P^3$, where
$S_5$ is the quartic.  One of the $S_i$, say $S_4$, meets $S_5$ in four
lines; the other three meet it in three lines of which one is double.
The three lines are given by $x_1+x_2 = x_i = 0$ for $i = 0, 1, 3$.

\begin{remark}\label{k3-fib}
Write the equation of $H_3$ in the form $t^2 = q_1q_2$, where
$q_1, q_2$ are polynomials of degree $4$.  Then the general fibre of the map
$H_3 \to \P^1$ given by the two sets of equations $(t:q_1), (q_2:t)$
is the union of two K3 surfaces, so by passing to a double cover of the base
we represent a variety birational to
$H_3$ as a family of K3 surfaces with base $\P^1$.
\end{remark}

\begin{defn}\label{bl-p3}
Map $\P^3 \to \P^3 \times ({\P^1})^3$
by the identity on the first component and the
projections away from the three lines
$L_i: x_1+x_2 = x_i = 0$ mentioned above, and let
$B$ be the image.  Under the inverse map $B \to \P^3$, the 
points $L_i \cap L_j$ pull back to divisors, which we will denote $D_{ij}$.
The inverse image of $L_i$ then consists of $D_{ij}, D_{ik}$, and a third 
divisor, to be referred to as $D_i$.
\end{defn}

\begin{remark} For $i \ne j$ the three divisors $D_i, D_{ij}, D_j$ meet in an
ordinary double point, and these three points are the only singularities
of $B$.  There is an algebraic small
resolution of $B$, because the image in $B$ of the plane $x_1+x_2 = 0$ that
contains the three lines is smooth at the double points
(we use the criterion of \cite[Theorem 1.8]{meyer-diss}).
Since $B$ is a complete intersection defined by three equations of
multidegree $(1,1,0,0), (1,0,1,0), (1,0,0,1)$, 
the adjunction formula gives its canonical class as $(-1,-1,-1,-1)$.
\end{remark}

\begin{defn} Let the $R_i$ be the images in $B$ of the $S_i$.  Let $D_B$ be the
fibre product $H_3 \times_{\P^3} B$.  It is a double cover of $B$ branched
along the union of the $R_i$.
\end{defn}

The following statements would be challenging to prove by hand, but are
easy with the help of a computer.

\begin{prop}
The $R_i$ are surfaces of which no three intersect in a curve.  
Except for
$R_5$ which has three ordinary double points, they are smooth.  They do not
contain any of the singularities of $B$.  Their union is defined in $B$ by
a single equation of multidegree $(2,2,2,2)$.  Finally, $R_4$ is a blowup
of $\P^2$ in three points, while $R_1, R_2, R_3$ are blowups of $\P^2$ in two
points. \qed
\end{prop}

\begin{prop}\label{twenty-eight}
  There are $28$ smooth rational curves $C_i$ along which two of the $R_i$
  intersect, and all of 
their intersections are transverse.  No three of the $R_i$ intersect
in a curve, but there are $7$ points where three $R_i$ meet and $3$ points
where four $R_i$ meet.  These points are the intersections of three or six
$C_i$ respectively.  These $28$ curves include $19$ that map to lines in $\P^3$
and $9$ that map to points. \qed
\end{prop}

The conditions $h^{1,0} = h^{2,0} = 0, K = 0$ that are a 
part of the definition of a Calabi-Yau threefold (Definition \ref{def-cy}) 
are easy to verify for $B$ or for a resolution of its singularities.

\begin{prop}\label{easy-cy} Let $T$ be a smooth rational threefold and let 
$\pi: D \to T$ be a double cover branched along a smooth divisor $R$ of 
class $-2K_T$.  Then $h^{1,0}(D) = h^{2,0}(D) = 0$ and $K_D$ is trivial.
\end{prop}

\begin{proof} The statement on $K_D$ is an immediate consequence of the
Riemann-Hurwitz formula $K_D = \pi^*(K_T + R)$.  For the $h^{i,0}(D)$ we 
proceed as in
the proof of \cite[Theorem 2.1]{cynk-szem}.  We know that
$h^i(\O_D) = h^i(\O_T) + h^i(K_T)$, as mentioned at the beginning of
\cite[Section 1]{cynk-szem}.
But by Serre duality $h^i(K_T) = h^{3-i}(\O_T)$, and $h^i(\O_T)$ is a birational
invariant which is $0$ for $T = \P^3, i \in \{1,2\}$.
\end{proof}

We are now ready to prove:
\begin{thm}\label{is-cy} There is a smooth Calabi-Yau threefold
$CY_{13}$ of Euler characteristic $98$ which is birational to the variety $Q_1$
obtained by reduction from the graph polynomial of the graph labeled $(4,13)$
in \cite[Figure 5]{mf-qft}.
\end{thm}

\begin{proof} 
We have shown above that $Q_1$ is birational to the double cover of $B$
branched along the union of the $R_i$.  Numerically this is a Calabi-Yau;
to show that it actually is one, we use the method of Cynk-Szemberg 
\cite{cynk-szem}.
First we resolve the three singular points of $B$, obtaining a smooth threefold
with $h^{1,1} = h^{2,2} = 7, h^{0,0} = h^{3,3} = 1$, and all other Hodge numbers 
$0$.  Then we blow up the fourfold intersection points of the $R_i$,
which are smooth on $B$, to exceptional divisors $E_j$.  When we do this,
we increase the canonical divisor on $B$ by $2 \sum E_j$; on the other hand, 
we replace the four surfaces by their strict transforms, 
and their classes decrease by $E_j$.
We thus preserve the relation $2K_B + \sum |R_i| = 0$.  We also increase the
Euler characteristic $\chi$ of $B$ by $2$, because we replaced a point with a 
$\P^2$.  The four surfaces that contain the point see it replaced by a rational
curve and thus their Euler characteristic increased by $1$.

Then we blow up the $28$ curves of intersection.  Each time we do this, we
obtain an exceptional divisor $E_C$ which we add to the canonical divisor, but
we decrease the classes of the two surfaces containing the curve by $E_C$, so 
again the relation $2K_B + \sum |R_i| = 0$ is preserved.  If the curve meets
a third surface in a point, this blows up that point on that surface, 
increasing its Euler characteristic by $1$.  Also we replaced
the rational curve, which has $\chi = 2$, by a $\P^1$-bundle over it
with $\chi = 4$.
Finally, we take a small resolution of the double points above the three
singular points of $R_5$, which are in the branch locus.  (This may take us
out of the projective category, but that does not matter here.)  This adds $3$
to the Euler characteristic of the branch locus and of the base.  We still
have $2K_B + \sum |R_i| = 0$, and now the base and branch locus are smooth.
In view of Proposition \ref{easy-cy}, this shows that it is a Calabi-Yau.

The Euler characteristic of the base is now $16 + 2 \cdot 3 + 2 \cdot 28+3=81$.
The branch locus originally consisted of surfaces whose Euler 
characteristics added to $42$.  However, we blew up $3 \times 4$ 
(for the fourfold points) $ + 7$ (for the threefold points) 
$+ 3$ (the singularities of the K3), so now we have $64$.  Their images in 
the blowup of the base are disjoint so the Euler characteristics add.
Thus the double cover has $\chi = 2 \cdot 81 - 64 = 98$.
\end{proof}

\begin{remark}\label{maybe-red}
The results of \cite{cynk-szem} are stated under the assumption 
that the intersection of any two components of the branch locus is
irreducible.  However, the arguments are entirely local,
so this assumption is unnecessary.
\end{remark}

\begin{remark} The fibration introduced in Remark \ref{k3-fib} has fibres
  whose transcendental lattice is isomorphic to $U \oplus \langle -26 \rangle$,
  where $U$ is the hyperbolic lattice with Gram matrix
  $\begin{pmatrix}0&1\cr1&0\cr\end{pmatrix}$.  Since it is not isotrivial,
    Theorem \ref{is-cy} together with \cite[Theorem 2.13]{dhnt} shows that
    $CY_{13}$ is a rigid
    Calabi-Yau threefold.  However, we give a different proof,
    because the geometry of $CY_{13}$ that we describe is 
    useful for counting the points to identify the modular form and for
    discovering other rigid Calabi-Yau threefolds.  In addition, this method
    of proving rigidity does not seem to apply to our other examples.
\end{remark}

To start our proof that $CY_{13}$ is rigid,
we introduce the configuration of lines in which the $s_i$
intersect.

%%% I have to say where exactly I prove rigidity!
\begin{defn} A {\em $(4,13)$-configuration} in $\P^3$ consists of
four planes $P_i (1 \le i \le 4)$ in general position and $13$ lines 
$L_{ij} (1 \le i \le 4, 1 \le j \le \max(i,3))$ satisfying the following
conditions:

\begin{enumerate}
\item For all $i, j$, the line $L_{ij}$ is contained in $P_i$ and in no
other $P$.
\item The following are the maximal sets of lines that intersect in a point:
$$\begin{aligned}
&\{L_{11},L_{13},L_{21},L_{23}\},\{L_{11},L_{12},L_{31},L_{32}\},
\{L_{21},L_{22},L_{31},L_{33}\},\\
&\quad \{L_{11},L_{41},L_{44}\},\{L_{21},L_{42},L_{44}\},\{L_{31},L_{43},L_{44}\},\\
&\quad \{L_{12},L_{22},L_{43}\},\{L_{13},L_{33},L_{42}\},\{L_{23},L_{32},L_{41}\},\\
&\quad \{L_{12},L_{13}\},\{L_{22},L_{23}\},\{L_{32},L_{33}\}.
\end{aligned}$$
\item If the intersection point of one of these sets of lines is contained in
$P_n$, then one of the lines is $L_{nj}$.
\item The intersection $L_{44} \cap P_i \cap P_j$ is empty for all $1 \le i < j
\le 3$.
\end{enumerate}
\end{defn}

%% \begin{remark} It appears to be a coincidence that
%% $(4,13)$ can be taken as referring either
%% to the weight and level of the modular form or to the number of planes and lines
%% in the configuration.
%% \end{remark}

\begin{prop} The automorphism group of $\P^3$ acts simply transitively on the
set of $(4,13)$-configurations.
\end{prop}

\begin{proof}
Applying an automorphism, we take the $P_i$ to be the coordinate hyperplanes,
the sequence of which is stabilized by the diagonal subgroup of $PGL_4$.
Then by condition $4$, we rescale $x_1, x_2$ so that $L_{44}$ is given by
$x_0+x_1+x_2 = x_3 = 0$.  Now $L_{41}$ meets $L_{44}$ in a point where 
$x_0=0$, since the point is also on $L_{11}$, so it must be $(0:-1:1:0)$.
It is not the line $x_0=x_3=0$, so we write it as $ax_0+x_1+x_2=x_3=0$.  Similarly
$L_{42}, L_{43}$ are the lines $x_0+bx_1+x_2=x_3=0, x_0+x_1+cx_2=x_3=0$.

We take $L_{11}$, which is a line in $x_0=0$ through $(0:-1:1:0)$ which is
not $x_0=x_3=0$, so it is $x_0=x_1+x_2+dx_3=0$.  Similarly $L_{21}, L_{31}$ are given by
$x_1=x_0+x_2+ex_3=0, x_2=x_0+x_1+fx_3=0$.  From the fact that any two of these lines meet
it follows that $d=e=f$.

Now $L_{13}, L_{23}$ pass through $(0:0:-d:1)$, the point on $L_{11} \cap L_{21}$.
So they are respectively $x_0=gx_1+x_2+dx_3=0, x_1=hx_0+x_2+dx_3=0$.  From the fact that
$L_{13}$ meets $L_{42}$ we learn that $g=b$, and because $L_{23}$ meets $L_{41}$
it must be that $a=h$.  Similarly $L_{12}$ meets $L_{43}$ and goes through the
point where $L_{11}$ meets $L_{31}$, so it must be $x_0 = x_1+cx_2+dx_3=0$, and 
$L_{32}$ meets $L_{41}$ and also passes through this point, so it is
$x_2=ax_0+x_1+dx_3=0$.  Similarly $L_{22}$ is $x_1=x_0+cx_2+dx_3=0$ and $L_{33}$ is
$x_2=x_0+bx_1+dx_3=0$.

Now $L_{12}$ meets $L_{22}$ in the point $(0:0:d:c)$, even though this has
not been specifically required.  But we also want this point to be on
$L_{43}$, so $c = 0$.  Similarly, by considering the sets
$\{L_{13},L_{33},L_{42}\},\{L_{23},L_{32},L_{41}\}$, we see that $a = b = 0$.

The stabilizer of the part of the configuration we have considered
consists of automorphisms $(x_0:x_1:x_2:x_3) \to (x_0:x_1:x_2:\lambda x_3)$,
and since $x_3$ always appears on its own or in the monomial $dx_3$ we 
replace $x_3$ by $x_3/d$.  We now have no nontrivial automorphisms left and
no parameters appearing in the equations of the $13$ lines.  To recapitulate,
they are:
$$\begin{aligned}& L_{11}: x_0=x_1+x_2+x_3=0, L_{12}: x_0=x_1+x_3=0, L_{13}: x_0=x_2+x_3=0,\cr
&\quad L_{21}: x_1=x_0+x_2+x_3=0, L_{22}: x_1=x_0+x_3=0, L_{23}: x_1=x_2+x_3=0,\cr
&\quad L_{31}: x_2=x_0+x_1=0, L_{32}: x_2=x_1+x_3=0, L_{33}: x_2=x_0+x_3=0,\cr
&\quad L_{41}: x_3=x_1+x_2=0, L_{42}: x_3=x_0+x_2=0,\cr
&\quad L_{43}: x_3=x_0+x_1=0, L_{44}: x_3=x_0+x_1+x_2=0.\cr\end{aligned}
$$
\end{proof}

Using Magma we find that the linear system of octic surfaces singular along 
the $13$ lines of the $(4,13)$ configuration and the six lines $P_i \cap P_j$
has projective dimension $2$.
In fact, every such surface contains the union of the coordinate planes.
Thus we consider a 
net of quartics obtained by dividing out $x_0x_1x_2x_3$ from the sections of 
this linear system.  As a basis for this system, we may take
\begin{multline}\label{ls4-seqs}
V_0 = x_0x_1x_2x_3, \quad V_1 = (x_0+x_1+x_3)(x_0+x_2+x_3)(x_1+x_2+x_3)x_3, \cr
\quad V_2 = (x_0+x_1+x_2+x_3)(x_0+x_1+x_3)(x_0+x_2+x_3)(x_1+x_2+x_3).\cr
\end{multline}

\begin{defn} We denote the tangent bundle of a smooth variety $V$ by $\T_V$.
\end{defn}

\begin{prop}\label{twod-def}
  Let $V$ be a general quartic in the linear system spanned by
$V_0, V_1, V_2$.  Then the double cover of $\P^3$ with branch locus 
$VV_0 = 0$ has a resolution which
is a Calabi-Yau threefold with $h^{2,1} = 2$.
\end{prop}

\begin{proof} It is easily checked that the singular locus of $VV_0 = 0$
  is supported on the $(4,13)$ configuration and that it consists only of
  double curves and quadruple points (the latter at the six singular points
  of $V_0 = 0$).  In addition, all the curves in the singular locus are of
  geometric genus $0$ and any two meet transversely.

  Let $X$ be the double cover of $\P^3$ branched along $VV_0 = 0$.  Then as in
  \cite[Section 6]{cynk-vs} or \cite{cynk-szem}
  there is a Calabi-Yau resolution $\tilde X$ of $X$
  which is a double cover of a blowup $\tilde \P^3$ along a smooth branch locus.
  (We have already described the construction of a similar resolution in
  the proof of Theorem \ref{is-cy}.)
  We also observe that
  $h^1(\T_{\tilde \P^3} \otimes \tilde \L^{-1}) = h^1(\T_{\P^3}) = 0$
  by \cite[Proposition 5.1]{cynk-vs},
  where $\tilde L$ is the line bundle corresponding to the branch divisor.
  This tells us that every deformation of $\tilde X$ is a double cover of a
  deformation of $\tilde \P^3$.
  
  Combining this with \cite[Propositions 2.1--2.2]{cynk-vs} we conclude that
  $h^1(\T_{\tilde X})$ is the dimension of the space of equisingular deformations
  of the branch locus.  The $(4,13)$ configuration cannot be deformed, and the
  octics singular along its lines form a $\P^2$, so $h^1(\T_{\tilde X})$ is
  of dimension $2$.
  %% In general $h^{2,1}$ is the dimension of the space of infinitesimal
  %% deformations, but deformations of Calabi-Yaus in characteristic $0$
  %% are unobstructed.  (For an
  %% accessible exposition of this result, as well as references to the original
  %% papers, the reader may consult \cite{litt}.)  So this
  %% is the same as the dimension of the tangent space to the space of global
  %% deformations.  
\end{proof}

%% \begin{remark}\label{equi} This can also be proved by computing an equisingular
%% deformation ideal as in \cite{cynk-vs}.
%% \end{remark}

Thus, in order to find candidates for rigid Calabi-Yau threefolds in this
family, we need to find octics in the
linear system that have no deformations preserving their singularities.  In
other words, we need either a point in the $\P^2$ parametrizing the linear
system corresponding to an octic with a singularity that cannot be deformed,
or the intersection of two curves in the $\P^2$ corresponding to extra
singularities.

\begin{remark}
  As Meyer points out in a similar situation \cite[p{.} 161]{meyer-diss},
  the double covers branched along maximally singular octics,
  even if they are Calabi-Yau threefolds,
  are not necessarily rigid.  It is also necessary that the construction
  of the resolution not involve blowing up any curves of positive genus.
  I am grateful to Colin Ingalls for helping me to understand the ideas of
  \cite{cynk-vs} and to express the argument of Proposition \ref{twod-def}.
\end{remark}

To list these maximally singular octics,
we introduce an incidence correspondence $I \subset \P^3 \times \P^2$
such that $(p_3,p_2) \in I$ if and only if $p_3$ is a singular point of the
octic corresponding to the point $p_2$ in the linear system.  To be precise,
let the linear system of octics have basis $O_0, O_1, O_2$, where $O_i = V_0V_i$.
Then $I$ is defined by the partial derivatives
$\frac{\partial \sum_{i=0}^2 y_iO_i}{\partial x_j}$ for $0 \le j \le 3$.
Now, $I$ will have a component of dimension $3$ supported at $x=y=0$, 
reflecting the fact that every linear combination of $O_0, O_1, O_2$ is
singular along the line $x=y=0$; similarly for the other $18$ lines in the
singular subscheme of every octic in the family.  However,
we are only interested in the components
of $I$ that map non-dominantly to $\P^2$.  Of these there are $7$, but 
the image of one of them is a point contained in several others, which is
of no use to us.  The others are curves $C_i$ of degree $1,1,1,1,2,4$, of which
the quartic curve has one ordinary node and two ordinary cusps and is rational,
while the conic has rational points.  The octic corresponding to a general
point of $\P^2$ (resp{.} a point on one of the $C_i$, resp{.} a singular
point of $\cup C_i$) has $h^{2,1} = 2$ (resp.{} $1, 0$).
The $C_i$ are defined by the polynomials
\begin{multline}\label{surf-13}
 y_0, \quad y_1, \quad y_2, \quad y_0 - y_1, \quad 4y_0y_2 + y_1^2 - 2y_1y_2
+ y_2^2,\\
y_0^2y_1^2 + 16y_0y_1^3 + 64y_1^4 + 2y_0^2y_1y_2 + 144y_0y_1^2y_2 
        - 128y_1^3y_2 + y_0^2y_2^2 \\
\quad + 228y_0y_1y_2^2 + 96y_1^2y_2^2 + 
        104y_0y_2^3 - 32y_1y_2^3 + 4y_2^4.\\
\end{multline}

There are $11$ rational points of intersection of two or more $C_i$, one of 
which is a singularity of the quartic; the quartic has two other singular 
points.  Of these $13$ points, the four with $y_2 = 0$ give
octics which are multiples of $x_3^2$, so the double covers are not Calabi-Yaus.
Seven of the octics have singularities that lead to canonical singularities
on the double cover and so may give rigid Calabi-Yau threefolds.
Indeed we find Hecke eigenforms of weight
$4$ for which $a_p \equiv 1-[V]_p \bmod p$ in each case.
Of these, the three given by the points $(0:0:1),(0:1/2:1),(0:1:1)$ 
split completely into linear factors and duplicate known
arrangements of eight hyperplanes found in the table of
\cite[p{.} 68]{meyer-diss}.  One of them, coming from $(-1/4:0:1)$, is
the threefold we have been looking at, which appears to correspond to the 
form of level $13$.  The other three, given by 
$(-1/9:-1/9:1),(-1:1:1),(-256/3:-5/6:1)$, seem to realize the forms 
$78/4, 10/1, 390/5$ in the indexing used by Meyer.  Apparently no known
Calabi-Yau was previously conjectured to correspond to $78/4$, while there
are known threefolds giving $10/1$.  For $390/5$, the table of
\cite[p{.} 116]{meyer-diss} shows a nonrigid threefold;
however, we have found a rigid
one, as we will show in Proposition \ref{rigid390}.
%% Although we will not be able to prove that our threefold does in fact
%% realize $390/5$, much less would be required to complete our calculation to a
%% proof than Meyer's.

\begin{remark} In addition to the rational points, there are three 
Galois orbits of nonrational intersection points, one each with points defined
over $\Q(\sqrt{3}), \Q(i)$, and the cubic field of discriminant $148$.  
%% One expects the
%% numbers of points on the Calabi-Yaus given by the first and third of these to 
%% be governed by the traces of eigenvalues of Hilbert modular forms over the 
%% respective fields, but it is not clear how to characterize the second one.
\end{remark}

The octics corresponding to $(-64:-3/2:1)$, $(-1:3:1)$ have isolated
singularities of type $A_2$, $D_4$.  The singularity at an
$A_2$ point of the branch locus 
has no crepant resolution % (\cite{reid-min}, corollary 1.16),
so there is no reason to expect
a corresponding Calabi-Yau, and there may be no
modular form matching the point counts.
On the other hand,
the singularity at a $D_4$ point does admit crepant resolutions in some cases.
Both of these statements follow from a result due to Shepherd-Barron
(\cite[Proposition 5.1]{iyama-wemyss}).
Since we have not found a modular form corresponding to the octic given by
$(-1:3:1)$, we expect that this is not such a case and thus that no Calabi-Yau
resolution exists.  It is also possible that the level of the modular form is
beyond the range of
%% However, no modular form has turned up,
%% perhaps because the level is too large: the bad primes are $2, 3, 5, 23$,
%% and new singular components of dimension $1$ appear on reducing mod $2$ or $3$.
%% Based on \cite[Conjecture 6.3]{meyer-diss}, we might expect the level to 
%% be of the form $2^a\cdot 3^b\cdot 5^c\cdot 23$, where $a, b \ge 1, c \le 1$.
%% If so, it is not unlikely that the form would be beyond the range of
\cite[Appendix C]{meyer-diss}.
%% Another possibility is that the
%% singularity is locally factorial and so admits no crepant resolution.

%% \begin{question} Meyer \cite{meyer-diss}, section 4.8 gives certain surfaces
%% in $\P^3$ consisting of
%% the union of a cubic and five planes whose double covers have point counts
%% suggesting that they correspond to the modular form $10/1$.  However, those
%% cubics are not isomorphic to the one we have found here.  Based on the
%% Tate conjecture, we expect a correspondence between his double covers and ours
%% (compare \cite{meyer-diss}, conjecture 1.10).  Can it be found?
%% \end{question}

We now have enough information to prove that $CY_{13}$ is rigid.  As in
\cite{gouvea-yui}, \cite{dieu-man} this means that it is modular.
\begin{thm}\label{rigid13} $CY_{13}$ is rigid.
\end{thm}

\begin{proof} The proof is almost identical to that of Proposition
  \ref{twod-def}.  The octic surface $V_0(-V_0/4+V_2) = 0$ is maximally
  singular, so there are no deformations as a double cover.  On the other
  hand, the base
  $B$ of the Cynk-Szemberg resolution constructed in the proof of
  Theorem \ref{is-cy} was obtained from $\P^3$ by blowing up only points and
  rational curves, so $H^1(\T_B \otimes K_B) = 0$.
  As in the proof of Proposition
  \ref{twod-def}, we conclude by referring to
  \cite[Propositions 2.1--2.2]{cynk-vs}.
\end{proof}

\section{Counting points}
In this and the next section
we will prove that the Galois representation on $H^3$ of the
rigid Calabi-Yau threefold $CY_{13}$ constructed
in the proof of Theorem \ref{is-cy} is isomorphic to that associated to the
rational newform $f_{13,4}$ of level $13$ and weight $4$.
In order to do this, we need to count the $\F_p$-points of $CY_{13}$.
We do not have an explicit model for
$CY_{13}$, so we will show how to relate its point counts to those of $H_3$
and count points there instead.  
%% Since $Q_3$ is a quintic threefold in $\P^4$, it is easy
%% to count its $\F_p$-points for reasonably small primes (in fact we will only 
%% need $p \le 50$).  Doing so, we find the following.
%% \begin{prop}\label{count-q3} For all primes $5 \le p \le 200$ the number of
%% $\F_p$-points on $Q_3$ is equal to $p^3 + 15p^2 - 26p + 1 - a_p$, where
%% $a_p$ is the Hecke eigenvalue for $f_{13,4}$.
%% \end{prop}
%% It is also easy to count points on the double octic $H_3$ associated
%% to $Q_3$.  
Using a computer the following statement is easily verified:
\begin{prop}\label{count-doct} For all primes $5 \le p \le 200$ we have
$[H_3]_p = p^3 + 6p^2 - 15p + 1 - a_p$, where $a_p$ is the $T_p$-eigenvalue
for $f_{13,4}$. \qed
\end{prop}

\begin{remark}
One could prove the relation 
$[Q_3]_p - [H_3]_p = 9p^2 - 11p$
for all $p \ge 5$ using the birational equivalence given in Proposition
\ref{to-hyp}, but this seems unnecessary since the
calculations needed here can easily be done on $H_3$.
\end{remark}

We now relate the number of points on $CY_{13}$ to that on $H_3$.

\begin{thm}\label{count-rigid} For all primes $p>2$ the relation
$[CY_{13}]_p - [H_3]_p = 43p^2 + 64p$ holds, where $CY_{13}$ is the rigid 
Calabi-Yau threefold of Theorem \ref{is-cy}.
\end{thm}

The rest of the section will be devoted to the proof of this theorem.
To begin, we pass from $H_3$ to the double cover $D_B$ of the variety $B$ 
defined in Definition \ref{bl-p3}.  

\begin{prop}\label{count-octic} For all $p>2$ we have 
$[B]_p = p^3 + 7p^2 + 4p + 1$ and
$[D_B]_p = [H_3]_p + 12p^2 - 3p$.  If $D_B'$ is the small resolution of $D_B$ at 
the double points above those of $B$, then $[D_B']_p = [D_B]_p + 6p$.
\end{prop}

\begin{proof} The map $B \to \P^3$ is an isomorphism outside the union of
the $D_i, D_{ij}$.  These are all isomorphic to $\P^1 \times \P^1$, so together
they
contain $6(p^2+2p+1)$ points.  There are nine nonempty twofold intersections:
the six $D_i \cap D_{ij}$ are lines containing $p+1$ points each, and the three
$D_i \cap D_j$ are points.  Finally, the three nonempty
threefold intersections $D_i \cap D_{ij} \cap D_j$ are points.  
So the total number of points in the union is 
$$6(p^2+2p+1) - 6(p+1) - 3 + 3 = 6p^2 + 6p.$$  These map to the union of three
coplanar lines in $\P^3$, which has $3p$ points, so $[B]_p = [\P^3]_p + 6p^2 + 
3p$ as claimed.

The next step is to compare $D_B$ to $H_3$.  On all six of the divisors,
the function 
defining the double cover is identically the square of the product of
a $(1,0)$-form and a
$(0,1)$-form on $\P^1 \times \P^1$.
So the double cover on a divisor consists of one point above each of the
$2p+1$ in the branch locus, and two above each of the $p^2$ remaining points.
In addition, on the intersection of two adjacent
branch divisors, the function is the square of a linear form, and on the
singular points it is a nonzero square.  Further, one of the lines in the
branch locus of a component misses the branch loci of the adjacent components,
while the other one meets them in one point each.

Thus, of the $6p^2+6p$ points of the union of the $D_i, D_{ij}$, we find that
$12p$ of them are in the branch locus of the double cover $D_B \to B$, while
the rest pull back to $2$ points each, giving a total of $12p^2$ points.
Since these correspond to $3p$ points of $H_3$ above the base locus of the map 
$\P^3 \to B$, the claim that $[D_B]_p = [H_3]_p + 12p^2 - 3p$ is proved.

For the last statement, note that the three double points of $B$ are
not in the branch locus, so they give two double points of $D_B$ each.  On
$B$, the small resolutions give rational curves with points
over $\Q$, so $p+1$ points over $\F_p$.  As already mentioned the branch 
function is $1$ there, so each component pulls back to two components which
curves of genus $0$ with rational points.  This means that we replace the
$6$ points with $6$ curves having $p+1$ points each.
\end{proof}

\begin{defn}
Let $D_V \to V$ be a double cover of a smooth threefold with branch divisor 
$R_V$ consisting of smooth components $B_1, \dots, B_k$.  Let $C_1, \dots, C_l$
be the components of the intersection of two or more of the $B_i$.
Assume that no curve
is contained in more than $3$ of the $B_i$, that no point is contained in more
than $5$ of the $B_i$, that the intersection of any two of the $C_j$ is
transverse, and that the relation $\sum_i [B_i] = -2K_V$ holds in $\Pic V$.
We refer to
the resolution of $D_V$ constructed in \cite{cynk-szem} as the 
{\em Cynk-Szemberg resolution} of $D_V$.
\end{defn}

\begin{remark} The hypotheses on the $B_i$ are those of
  \cite[Theorem 2.1]{cynk-szem},
except that we do not assume the $B_i \cap B_j$ to be irreducible
(cf.{} Remark \ref{maybe-red}).
\end{remark}

Before continuing with the proof of Theorem \ref{count-rigid},
we pause to explain how to count the points
on a double cover of an exceptional divisor without writing down the double
cover of the blowup explicitly.

\begin{meth}\label{count-points-dc}
A double cover is defined by a branch function
up to squares, so to find the double cover on the exceptional divisor we need
to choose a branch function that extends to the blowup without being identically
$0$ or $\infty$ there.  If we blow up a point, this is usually easy.  For
example, if the point is in $\P^3$, let it be $(c_0:c_1:c_2:c_3)$ and let
$y_0,y_1,y_2$ be variables: then evaluate the branch function at 
$(c_0+y_0:c_1+y_1:c_2+y_2:c_3)$ and discard all terms not of minimal degree in
the $y_i$.  The result is then the branch function for the exceptional $\P^2$
above a point.  More generally, if we are on another threefold, then instead
of adding multiples of $(1:0:0:0)$, etc., we simply add multiples of a basis
for the tangent space at the point.

For the double cover $\tilde D_B \to \tilde B$
of the blowup of a threefold $B$ along a curve $C$,
we proceed as follows.  
Write down the blowup $\tilde B$ of $B$ along $C$ and choose a point $P$ in the
exceptional divisor $E$ not contained in the branch locus.  Then find a curve
$C_P \subset \tilde B$ not contained in $E$ and smooth at $P$.
We know the branch function $f$ for $D_B \to B$, so restrict it to
the curve $C_P$ and use it to define a map $\alpha: C_P \to \A^1$ given by
$f/r^2$,
where $f,r^2$ are of the same degree and vanish to the same order at $P$ along
$C$.  Because $C_P$ is smooth at $P$, this map extends to $P$; it is not
$0$ or $\infty$ there, and the two points of the double cover lying above 
$P$ are defined over $\Q(\sqrt{\alpha(P)})$.  On the other hand, write down a
section $s$ of a suitable linear system on $E$ whose zero locus is the branch 
locus on $E$, and a form $t$ not vanishing at $P$ and defining a map 
$\beta: s \to \P^1$.  The correct branch function must be a multiple of $s$,
which has the correct branch locus, and must take $P$ to a square multiple of
$\alpha P$, so we have determined it up to an unimportant square factor.

In either case, we obtain a double cover of a rational surface as the 
exceptional divisor on $\tilde D_B$, which for us will always be a rational
surface.  The points may be counted in either of
two ways.  One way is to parametrize the surface and explicitly determine
the points blown up and curves contracted on the way from $\P^2$ to the
exceptional divisor.
%% ; since a smooth rational curve over $\F_p$ has $p+1$ points,
%% it is easy to count the points on the exceptional divisor.
Another, for the
exceptional divisors above a curve, is to use the natural projection to 
$\P^1$.  In our situation the branch locus will always consist of two sections
and some fibres.  The fibres in the branch locus are covered by rational curves
mapping isomorphically to them.  Where the two sections do not meet, the fibre
is replaced by a double cover of $\P^1$ branched in two points, which is
isomorphic to $\P^1$.  Where they do meet, the branch function on the fibre
is of the form $cf^2$ for a constant $c$ and the fibre is replaced by two curves
defined over $\Q(\sqrt c)$.  The reduction of this mod $p$ has $(1+(c/p))p+1$
points, as opposed to the $p+1$ points of the fibre in the exceptional divisor
on $B$.
\end{meth}

We now follow the steps of the Cynk-Szemberg resolution of $D_B$, keeping
track of the numbers of points at each stage.  The first step is to blow up
the fourfold points.

\begin{prop}\label{count-four} Blowing up each of the fourfold points on $B$
adds $p^2+2p$ points to the double cover over $\F_p$ for all odd $p$.
\end{prop}

\begin{proof}
Blowing up a fourfold point on $B$ replaces it by a
$\P^2$, and the four surfaces contribute four lines to the branch locus, any
two of which meet transversely.  To find the constant, we evaluate the
branching function at a general point of $B$ infinitely near to the fourfold
point.  This shows that the double covers are obtained by extracting 
square roots of
$$4(u^2vw + uv^2w + uvw^2 + v^2w^2), 4(-u^2vw - uv^2w - uvw^2),
4(u^2vw + u^2w^2 + uv^2w + uvw^2).$$
Meyer shows (\cite[p.{} 73]{meyer-diss}) that 
the surface in $\P(2,1,1,1)$ defined by $t^2 = \alpha uvw(u+v+w)$ over $\F_p$
has $p^2+\left(1+\left( \frac{-\alpha}{p}\right) \right)p+1$ points.
When we change
variables to reduce the double covers above to this form, we find that
$-\alpha = 4$ for each of them.
\end{proof}

\begin{defn} We introduce notation for the $28$ singular curves of $B$.
The $13$ that
map to the $L_{ij}$ of the $(4,13)$ configuration will be denoted 
$\tilde L_{ij}$, while the $6$ that map to the intersections of two of the
planes will be the $P_{ij}$.  The $9$ that map to a point in $\P^3$ will be
called $Q_i$.
\end{defn}

\begin{thm}\label{all-pos} Let $\tilde D_B$ be the blowup of $D_B$ along $C$,
where $C$ is a curve in the intersection of two of the $S_i$.
Then for all $p>2$ the exceptional
divisor of $\tilde D_B$ has $p^2 + (2+d)p + 1$ points over $\F_p$, where
$d$ is the number of points where $C$ meets exactly one other curve in the
branch locus.
\end{thm}

\begin{proof} For each curve $C$ we compute the branch locus and the branch
function for the double cover of the exceptional divisor of the blowup of
$C$ in $B$.  The following table describes the results.  Here
$F_i$ denotes the Hirzebruch ruled surface with special section of
self-intersection $-i$ and the classes are given in terms of $S, F$, the
classes of the special section and a fibre.  While the Euler characteristics
are not necessary for the argument, they are very useful for consistency checks.

\begin{table}[h]
\centering
\begin{tabular}{|l|c|c|c|}
\hline
Curves&Exc.{} div.&Branch classes&$\chi$(branch locus)\cr\hline
$P_{12}, P_{13}, P_{23}$&$F_0$&$S,S,F,F$&$4$\cr\hline %1,2,9
$\tilde L_{ij} (1 \le i \le 3, 2 \le j \le 3)$ & $F_0$ & $S,S+2F$ & 2\cr\hline %4,5,11,12,17,18
$P_{14}, P_{24}, P_{34}$ & $F_1$ & $S, F, S+F$ & 4 \cr\hline %3,10,16
$\tilde L_{11}, \tilde L_{21}, \tilde L_{31}$ & $F_1$ & $F, S+F, S+2F$ & 2 \cr\hline %7,15,20
%% branch polynomial is -x^2yz - xy^2z + x^2z^2/4 = xz(-xy-y^2+xz/4)
%% by changes of variables, all 9 in the next case reduce to this
$Q_i$ & $F_1$ & $F, S+F, S+2F$ & 2 \cr\hline %6,8,13,14,19,21,26,27,28
%% branch poly -3/4*x^2*y^2 + y^4 - x^2*y*z + 3*y^3*z + 3*y^2*z^2 + y*z^3,
%% or better (y+z)(x^2y-3x^2z+4z^3)
%% components meet in [1,0,0],[1,0,-1],[1,0,1], while the cubic has a cusp
%% at [0:-1:1] and the tangent line is x=0
$\tilde L_{41}, \tilde L_{42}, \tilde L_{43}$ & $F_1$ & $S+F, S+3F$ & 1\cr\hline%22,23,24
%% the image of F_3 in P^4 is defined by -v^2+yw+vw = xv-vw+w^2 =
%% xy-v^2+2vw-w^2 = 0, and the branch polynomial is z^2+6zv+5yw-6zw-5vw+5w^2.
%% the double cover goes nicely into P^5 and has p^2+3p+1 points.
%% Since the branch polynomial is 1 at the singular point (0:0:1:0:0), we
%% get p^2+5p+1 points as expected.
$\tilde L_{44}$ & $F_3$ & $S+3F, S+3F$ & 1\cr\hline %25
\end{tabular}
\smallskip
\caption{Exceptional divisor and branch locus above the 
intersection of two components of the branch locus of $D_B \to B$.}
\end{table}

Every double cover of $\P^1 \times \P^1$ branched along two fibres and two
sections is isomorphic to $\P^1 \times \P^1$, and thus has $p^2 + 2p + 1$
points over $\F_p$ for all $p$.  Also, these curves do not contain any point 
where exactly two of the $28$ curves intersect.

A double cover of $F_1$ branched along curves of class $S, F, S+F$
is the blowup of the cone over a smooth conic at the vertex.  This is $F_2$,
which has $p^2+2p+1$ points.

In the other cases, we apply Method \ref{count-points-dc} to construct the
double covers.  To count their points, we
check that the variety is birational to $\P^2$
in such a way that only curves that are isomorphic to $\P^1$ over $\Q$ and have
good reduction at all odd primes are blown up or down (otherwise the number
of points would not always change by $dp$).  

For example, for $L_{12}$, we find that the double cover of $\P^1 \times \P^1$
is defined by $$t^2 = -x_1^2y_1y_1 - x_1x_2y_1y_2 + x_2^2y_2^2/4.$$
Embedding $\P^1 \times \P^1$ into $\P^3$ by the Segre embedding and pulling
back, we get the surface in $\P^4$ defined by
$$x_1^2+x_2x_3+x_2x_5-x_5^2/4 = x_3x_4-x_2x_5 = 0.$$
Projecting away from the singular point $(0:-1:0:1:0)$ gives a birational 
equivalence with the smooth quadric $y_1^2 + y_2y_3 - y_4^2/4$, which has 
$p^2+2p+1$ points over $\F_p$ for all odd $p$.  This blows up the base point
to a smooth rational curve and contracts three smooth rational curves meeting
in the base point to the points $(-1/2:0:-1:1),(1/2:0:-1:1),(0:1:0:0)$, so 
the number of points on the surface in $\P^4$ is $2p$ more than that of the
quadric, which is $p^2+4p+1$ as claimed.

The other cases are similar.  Rather than working directly with the double
cover of $F_1$ or $F_3$, it is easier to map to $\P^2$ and restore the 
contribution of the inverse image of the exceptional divisor at the end.
If the branch locus meets the exceptional divisor in two distinct points,
it is covered by a $\P^1$ which has $p+1$ points.  If the two points coincide,
then there are $2p+1$ if the branch function restricts to a square, as is true
in all cases here.
%% The details are too boring for the paper.
\end{proof}

\begin{defn} Let $B_f$ be $B$ blown up on the fourfold points and let
$D_{B_f}$ be its double cover.
To avoid notational clutter, we will use $C, S_i$, etc.{} equally to refer
to curves and surfaces in $B, D_B$ and their strict transforms in $B_f, D_{B_f}$.
\end{defn}

%% To find the double cover exactly
%% without having to blow up $D_{B_f}$ explicitly, we proceed as follows.  
%% Write down the blowup $\tilde B$ of $B$ along $C$ and choose a point $P$ in the
%% exceptional divisor $E$ not contained in the branch locus.  Then find a curve
%% $C_P \subset \tilde B$ not contained in $E$ and smooth at $P$ (this is not
%% difficult).  We know the branch function $f$ for $D_B \to B$, so restrict it to
%% the curve $C_P$ and use it to define a map $\alpha: C_P \to \A^1$ given by
%% $f/r^2$,
%% where $f,r^2$ are of the same degree and vanish to the same order at $P$ along
%% $C$.  Because $C_P$ is smooth at $P$, this map extends to $P$; it is not
%% $0$ or $\infty$ there, and the two points of the double cover lying above 
%% $P$ are defined over $\Q(\sqrt{\alpha(P)})$.  On the other hand, write down a
%% section $s$ of a suitable linear system on $E$ whose zero locus is the branch 
%% locus on $E$, and a form $t$ not vanishing at $P$ and defining a map 
%% $\beta: s \to \P^1$.  The correct branch function must be a multiple of $s$,
%% which has the correct branch locus, and must take $P$ to a square multiple of
%% $\alpha P$, so we have determined it up to an unimportant square factor.

So far we have described the result of blowing up a single curve on $D_{B_f}$;
now we need to understand what happens when the curves are blown up in order.
In general, let $D_T \to T$ be a double cover of a threefold branched along
some surfaces $S_i$.  Let $C$ be a component of $S_1 \cap S_2$, 
and suppose that $C$ is smooth of genus $g$, meets all other components of 
the intersection of two $S_i$
transversely, and is disjoint from the singular loci of the $S_i$ and all
intersections of four of the $S_i$.  Let $\tilde D_T, \tilde T$
be the blowups of $D_T, T$ along $C$, with exceptional divisors $E_{D_T}, E_T$, 
and let $\tilde S_1, \tilde S_2$ be the
strict transforms of $S_1, S_2$.  Let $E_i = E_{D_T} \cap \tilde S_i (i \in \{1,2\})$.
Then $\tilde D_T$ is a double cover of $\tilde T$, as in
\cite[proof of Theorem 2.1, case (d)]{cynk-szem},
and the branch locus is the strict transform of
that of $D_T \to T$.  
Moreover, our assumptions ensure that $E_T$ is a $\P^1$-bundle over $C$ and 
that the $E_i$ are sections.  If we restrict $\tilde D_T \to \tilde T$ to a
double cover $E_{D_T} \to E_T$, the branch locus consists of the $E_i$ together
with the fibres where $C$ meets another $S_i$.

\begin{prop}\label{char-dc}
With the notation of the previous paragraph, let $d = \#(E_1 \cap E_2)$.  Then
the Euler characteristic of $E_T$ is $2(2-2g)$, while that of $E_D$ is
$2(2-2g)+d$.
\end{prop}

\begin{proof} Since $E_T$ is the exceptional divisor of a blowup of $T$ along
a smooth curve of genus $g$, it is a $\P^1$-bundle and hence has Euler 
characteristic $\chi(\P^1) \chi(C_g) = 2(2-2g)$.  To calculate $\chi(E_D)$,
we recall the formula $\chi(E_D) = 2\chi(E_T) - \chi(B)$, where $B$ is the
branch locus, which is valid for any double cover and easily seen from topology.
Here we have $\chi(E_1 \cup E_2) = 2(2-2g)-d$.  But $B$ is obtained from
$E_1 \cup E_2$ by adding fibres, which are $\P^1$ with two points deleted
and have Euler characteristic $0$, so $\chi(B) = 2(2-2g)-d$ as well and
$\chi(E_D) = 2(2-2g+d)$.
\end{proof}

\begin{prop}\label{meet-point} With the same notation as before, let 
$S_1, S_2$ be components of the branch locus of $D \to T$, and let
$C_1, C_2$ be components of $S_1 \cap S_2$ meeting transversely in a point $P$
on no other curve in the branch locus.  Let $D_{12}$ be $D$ blown up
along $C_1$ and then $C_2$ and let the exceptional divisors be $E_{11}, E_{12}$.
Let $D_2$ be $D$ blown up along $C_2$, and let the exceptional divisor
be $E_2$.  Suppose
that the Picard group of $E_2$ is defined over $\Q$.  Then, for all
primes $q$ of good reduction, we have $[E_2]_q - [E_{12}]_q = q$, and the
Picard group of $E_{12}$ is also defined over $\Q$.
\end{prop}

\begin{remark} We showed in Theorem \ref{all-pos} that the hypothesis on the
Picard group always holds for the resolution of $CY_{13}$.  If it did not
we would obtain the more general relation $[E_2]_q - [E_{12}]_q = (\alpha/q) q$,
where the field of definition of the two components of the fibre is
$\Q(\sqrt \alpha)$.  This will be used in a later example.
\end{remark}

\begin{proof} 
The normal directions to $C_1$ at $P$ in $S_1, S_2$ coincide with the tangent
direction to $C_2$, so 
$E_{11}, E_{12}, \tilde C_2$ all meet in a point $\tilde P$ above $P$.  Thus in 
the double cover $E_D \to E_T$, the two points in the branch locus coincide in 
the fibre above $P$ and so the fibre in $E_D$ becomes two curves of genus $0$
that intersect above the point.  The same thing would happen above $C_2$ if
we had blown it up first; but now the normal directions at $\tilde P$ in
$\tilde S_1, \tilde S_2$ no longer coincide, as that would imply a more
complicated singularity of $S_1 \cap S_2$ at $P$.  Thus the two points in
the branch locus in the fibre above $\tilde P$ remain distinct and we obtain
an irreducible curve.

Note further that the components in the fibre above $P$ in $E_D$ are the
only effective curves in their linear equivalence class, since they are
components of a fibre of a conic bundle.  This implies that
they are rational over $\F_p$ for all $p$ of good reduction, and in particular
that each has $p+1$ points.  

This applies equally well to the fibre above $P$ in the blowup of $C_2$, had
we blown up that curve first.  On the other hand, the fibre above $\tilde P$
in the blowup of $\tilde C_2$ has only $p+1$ points: thus, $p$ fewer than 
it would.  Since the two blowups are isomorphic if these fibres are deleted,
we conclude that the blowup of $\tilde C_2$ has $p$ fewer points than that of
$C_2$.  The statement about $\Pic E_{12}$ is now clear, since this group is
equal to the subgroup of $\Pic E_2$ generated by a section and all vertical
curves other than those above $P$.
\end{proof} 

%% Now suppose that we have two components $C_1, C_2$ of $S_1, S_2$ meeting
%% transversely at a point $p$.  Let us blow up $C_1$, obtaining $E_1, E_2$ as
%% before and $\tilde C_2$, the strict transform of $C_2$.  Then the normal 
%% directions to $C_1$ at $p$ in $S_1, S_2$
%% are the same, and coincide with the tangent direction to $C_2$ there, so 
%% $E_1 \cap E_2 \cap \tilde C_2$ contains a point above $p$.  This does
%% contribute to the Euler characteristic of $E_T$.  If these
%% directions are the same only to first order, however, %%% I need to check this!
%% then we have three curves of intersection of the branch locus meeting 
%% transversely at the point $p$, which does not affect the Euler characteristic.

%% To summarize, a transverse intersection of two singular curves at the branch
%% locus contributes $1$ to the Euler characteristic of the double cover of one
%% of the components, but nothing to the other one.  On the other hand, an 
%% intersection of three does not contribute, and there are no intersections of
%% four because they have already been blown up before we begin blowing up the
%% curves.  

We showed in Proposition \ref{count-octic} that $[D_B]_p - [H_3]_p = 12p^2$.
The singularities of the branch locus
consist of $28$ curves with $24, 7, 3$ intersection points of order $2,3,4$
and $3$ double points.  In Proposition \ref{count-four} we saw that blowing up
the fourfold intersections adds $3(p^2+2p)$ to the count, while Proposition
\ref{meet-point} together with Theorem \ref{all-pos} tells us that blowing up
the curves contributes $28(p^2+p) + 24p$.  Finally, we take a small resolution
of the $9$ ordinary double points ($2$ above each of the $3$ of $B$ and the
$3$ of the branch surface $S_5$), which adds $9p$ points.
We conclude that
$$[R]_p = [H_3]_p + 12p^2-3p + 3p^2+6p + 28(p^2+p)+24p+9p = [H_3]_p + 43p^2 + 64p$$
for all $p \ge 5$.
This completes the proof of theorem \ref{count-rigid}.
\qed

\section{Proving modularity}
Gouv\^ea-Yui and Dieulefait-Manoharmayum proved \cite{gouvea-yui}, 
\cite{dieu-man} that rigid Calabi-Yau
threefolds defined over $\Q$ are always modular. 
%% strictly we don't depend on this theorem but use Schutt to prove it
Later,
Dieulefait \cite{dieu} gave an algorithm to determine the
Hecke eigenform associated to a rigid Calabi-Yau.  However, his algorithm
requires knowing a basis for the space of newforms of weight $2$ and level
$\prod_i p_i^{a_i}$, where $p$ runs over the primes of bad reduction of the
threefold and $a_i$ is the bound on the valuation of the conductor of an
elliptic curve over $\Q$ at $p_i$.  For $Q_3$, for example, this bound would
be $2^8 3^5 13^2 > 10^7$, which is well beyond the range where such a basis
can be calculated.  So instead we use a method due to Serre and Sch\" utt
\cite{schutt}, inspired by ideas of Faltings.  We
will present the proof carefully, following \cite{schutt} closely, for the 
modular form of level $13$ and its corresponding Calabi-Yau, and then indicate 
the necessary changes to apply it in the other examples.

Sch\" utt assumes, following Serre, that we have two $\ell$-adic Galois 
representations $\rho_1, \rho_2$ unramified outside a finite set of primes $S$,
with the same determinant and for which the mod $\ell$ reductions are isomorphic
and absolutely irreducible.  In the applications, $\rho_1$ is the Galois
representation associated to a Hecke eigenform of weight $4$, while $\rho_2$
is the representation on $H^3$ of a rigid Calabi-Yau threefold (or more 
generally on $H^{3,0} + H^{0,3}$ of a Calabi-Yau which is not rigid but for which
this is nevertheless a component of $H^3_{\etale}$), and we take $\ell = 2$.  
For both representations, the determinant is the cube of the cyclotomic
character \cite{dieu-man}.  Also, $\rho_1$ is unramified away from $\ell$ and 
the primes dividing the level, while $\rho_2$ is unramified away from $\ell$
and the primes of bad reduction.

\begin{remark} As in \cite[Section 1.8.1]{meyer-diss},
  it is easy to determine $\tr \rho_2(\Frob_p)$ for $p \ge 17$
by counting points.  This is because its absolute value is at most $2p^{3/2}$,
while the sum of the traces on the other cohomology groups is of the form
$p^3+n(p^2+p)+1$.  So if $4p^{3/2} < p^2+p$ the point count determines $n$,
and this holds for $p \ge 17$.
\end{remark}

We now let $\rho_1$ be the representation attached to the newform of level
$13$ and weight $4$ and $\rho_2$ the representation on $H^3(R)$.  Our goal is
to prove the following theorem.

\begin{thm}\label{all-counts}
For all primes $p$ we have $\tr \rho_1(\Frob_p) = \tr \rho_2(\Frob_p)$.
\end{thm}

We start by checking this for small primes:
\begin{lemma}\label{small-counts}
For all primes $p \in [17,50]$ we have 
$\tr \rho_1(\Frob_p) = \tr \rho_2(\Frob_p)$.
\end{lemma}

\begin{proof} We consult \cite[Appendix C]{meyer-diss} for the
$\tr \rho_1(\Frob_p)$
and calculate the $\tr \rho_2(\Frob_p)$ by counting points and using Theorem
\ref{count-rigid}.
\end{proof}

In order to prove anything about $\rho_2$, we will need to know its ramified
primes, which requires us to determine a set containing the primes of bad 
reduction of $CY_{13}$.  These are contained in the set of primes of bad reduction of
$H_3$, which in turn are a subset of the primes contained in an associated 
prime of the ideal defining the singular subscheme of an integral model.

The following statement is presumably well-known; a proof is provided here
for lack of an adequate reference.

\begin{lemma}\label{bad-primes}
Let $I \subset \Z[x_0,\ldots,x_n]$ be an ideal and $B$ a Gr\" obner basis for 
$I$.  Let $p$ be a prime that does not divide the leading coefficient of any
element of $B$.  Then no associated prime of $I$ contains $p$.
\end{lemma}

\begin{proof}
By setting $x = 1, y = p$ in \cite[Exercise 15.22]{eisenbud}, 
we see that it suffices to show that 
$I = (I:p^\infty)$.  Suppose that $f_0 = p^nx \in I$.  Inductively choose 
$b_i \in B$ whose leading term divides that of $f_i$ and monomials $q_i$
such that $f_{i+1} = f_i - q_i b_i$ has leading term smaller in the monomial 
order than the leading term of $f_i$.  By hypothesis $p$ does not divide the 
leading coefficient of $b_i$, so $p^n$ divides the leading coefficient of $q_i$.
Since $f_0 \in I$, this process must terminate
in $0$ (\cite[Algorithm 15.7]{eisenbud}).
Thus we have written $f_0 = \sum_i q_i b_i$, where $b_i \in I$ and $p^n | q_i$.
It follows that $x \in I$ as desired.
\end{proof}

\begin{remark}
The converse of this statement is false; the set of primes dividing
the leading coefficient of an element of $B$ depends on the term order,
while the set of associated primes depends only on the ideal.
\end{remark}

Using Magma to find a Gr\" obner basis for the ideal generated by the partial
derivatives, we see that $H_3$ has good reduction away from 
$\{2,3,5,13\}$, and then we check that the reduction mod $5$ of the singular
subscheme of $H_3$ is equal to the singular subscheme of $H_3$ base changed
to $\F_5$, so that $5$ is in fact a prime of good reduction.  We conclude that
$\rho_2$ is unramified outside $\{2,3,13\}$.

Our first task is to determine the mod $2$ representations $\bar \rho_i$.
Note that the trace of $\alpha \in GL_2(\F_2)$ is $0$ if $\alpha$ has order
$1$ or $2$ and $1$ if $\alpha$ has order $3$.  Taking $\alpha = \Frob_p$,
this says that $\tr \Frob_p$ is $0$ if the residue field of a prime above
$p$ in the field cut out by $\bar \rho_i$ has order $p$ or $p^2$ and $1$ if
the order is $p^3$.

Also, $\alpha$ has a fixed
$1$-dimensional subspace if and only if its order is $1$ or $2$.  It follows
immediately that a representation
to $GL_2(\F_2)$ is absolutely irreducible if and only if its trace is
not identically $0$.  This condition holds in our example, because both
$a_5$ of the modular form and the number of points mod $5$ are odd.

\begin{defn} Let $F_3 = \Q[x]/(x^3-x-2)$ and let $F_6$ be its Galois closure.
\end{defn}

\begin{prop}
Both $\bar \rho_1$ and $\bar \rho_2$ are the projection
$\Gal(\bar \Q/\Q) \to \Gal(F_6/\Q)$ followed by an isomorphism
$\Gal(F_6/\Q) \cong GL_2(\F_2)$.
\end{prop}

\begin{proof}
The tables of Jones and Roberts \cite{tables}
give all extensions of $\Q$ with Galois group $C_3, \S_3$ unramified outside
$S$.  We find that $\tr \bar \rho_i(\Frob_p) = 1$ for $p = 17$ and
$0$ for $p = 19, 23, 29$.  The only choice for $\ker \rho_i$
consistent with this is $\Gal(\bar \Q/F_6)$.
\end{proof}

\begin{defn}
Let $\sigma_1, \sigma_2$ be two Galois representations to $GL_r(\Z_\ell)$
with the same determinant whose reductions mod $\ell^n$ are equal for some
$n>0$.  We define the {\em discrepancy function}
$d_n: \Gal (\bar\Q/\Q) \to \F_\ell$ by
$d_n(\alpha) = (\tr \sigma_1(\alpha) - \tr \sigma_2(\alpha))/\ell^n$.
Further, we define the {\em discrepancy representation}
$$\delta_n: \Gal (\bar \Q/\Q) \to M_r^{\tr = 0}(\F_\ell) \rtimes GL_r(\F_\ell)$$
as follows.  Let $\mu(\alpha)$ be the matrix such that
$\sigma_1(\alpha) = (1+\ell^n \mu(\alpha)) \sigma_2(\alpha)$.  Then
$\delta_n(\alpha) = (\mu(\alpha),\bar \sigma_1(\alpha))$.
\end{defn}

\begin{remark} These objects are defined in \cite[Section 5]{schutt}
(the only difference is in the use of subscripts to facilitate the
statement of Corollary \ref{lift-same} below),
but Sch\" utt does not give English names to them.
Of course $d_n$ is identically $0$ if and only if the traces are congruent
mod $\ell^{n+1}$.  It is difficult to work with $d_n$ directly because it does not
have useful algebraic properties.
\end{remark}

\begin{prop}[{\cite[Section 5]{schutt}}]
$\delta_n$ is a continuous homomorphism unramified at every prime
where $\sigma_1, \sigma_2$ are unramified, and
$d_n(\alpha) = \tr (\mu_\alpha \bar \sigma_1(\alpha))$.  In the
case of interest $\ell = r = 2$, the target group is isomorphic to
$\S_4 \oplus \cyc_2$, and $\tr (M_1M_2) = 1$ if and only if the order of
$(M_1,M_2)$ in the semidirect product is $4$ or $6$. \qed
\end{prop}

\begin{cor}\label{lift-same}
Let $\rho_1, \rho_2$ be Galois representations to $GL_2(\Z_2)$, unramified
outside $S$, whose 
residual representations are equal and surjective
and cut out the field $F$.  Suppose that
the traces of $\rho_1, \rho_2$ are congruent mod $\ell^n$.  Let
$E_F$ be the set of all extensions of $\Q$ with Galois group
$\S_4$ or $\S_3 \oplus \cyc_2$ that are unramified outside $S$ and contain $F$.
Suppose that for every $E_i \in E_F$ there is a prime $p_i$ of inertial degree
$4$ or $6$ such that 
$\tr \rho_1(\Frob_p) \equiv \tr \rho_2(\Frob_p) \bmod \ell^{n+1}$.  Then
the traces of $\rho_1, \rho_2$ are congruent mod $\ell^{n+1}$.
\end{cor}

\begin{proof}
Let $F'$ be the field cut out by $\delta_n$.  If the traces are not congruent
mod $\ell^{n+1}$, then $F' \ne F$.  But $F \subset F'$ and $\Gal(F'/\Q)$ is
a subgroup of $\S_4 \oplus \cyc_2$ mapping surjectively to $\S_3$, 
so as Sch\" utt shows it is isomorphic to one of 
$\S_4 \oplus \cyc_2, \S_4, \S_3 \oplus \cyc_2$.  In the first case, replace $F'$
by the subfield fixed by $\cyc_2$.  Now the Chebotarev density theorem gives
us a prime $q$ for which the order of Frobenius in $\Gal F'/\Q$ is $4$ or $6$
and hence $d_n(\Frob_q) = 1$.  However, our hypothesis gives us a prime
$q'$ with Frobenius of the same order and $d_n(\Frob_{q'}) = 0$.  Since 
$\S_4$ and $\S_3 \oplus \cyc_2$ have only one conjugacy class of order $4$ or
$6$, in fact $\Frob_q$ and $\Frob_{q'}$ are conjugate, a contradiction.
\end{proof}

\begin{remark} The proof is the same as that used by Sch\" utt to conclude
that $\tr \rho_1 = \tr \rho_2$ if the traces of Frobenius are the same for
the given set of primes.  The slightly more precise statement given here
is not needed in the present paper.
\end{remark}

\begin{prop}\label{seven-exts}
There are $7$ extensions of $\Q$ with Galois group $\S_4$ containing $F_6$
and unramified outside $\{2,3,13\}$.  For each one, at least one of
$19, 23, 41$ has Frobenius of order $4$.
\end{prop}

\begin{proof}
From the tables of \cite{tables} we find that there are $150$ extensions of
$\Q$ with Galois group $\S_4$ unramified outside $\{2,3,13\}$.  However, 
by calculating cubic resolvents we find that only $7$ of these contain $F_6$.
It is routine to verify the second statement, since the Frobenius in the
$\S_4$-extension has order $4$ if and only if the prime is inert in the
non-Galois quartic subfield.
\end{proof}
% x^4 - 4*x + 1, x^4 - 8*x^2 - 8*x + 2, x^4 - 2*x^3 + 8*x^2 + 6*x - 17,
% x^4 - 4*x^2 - 8*x - 6, x^4 - 52*x - 143, x^4 - 104*x - 78, 
% x^4 - 52*x^2 - 104*x + 234

\begin{prop}\label{unram} There are $14$ extensions of $\Q$ with Galois group 
$\S_3 \oplus \cyc_2$ containing $F_6$ and unramified outside $\{2,3,13\}$.  For
each one, at least one of $17, 37, 43$ has Frobenius of order $6$.
\end{prop}

\begin{proof} These extensions are obtained by adjoining a square root of the
product of a subset of $\{-1,2,3,13\}$; the subsets $\{\}, \{-1, 2, 13\}$ are
not usable because their products are squares of elements of $F_6$.
The Frobenius is of order $6$ if and only
if it has order $3$ in $\Gal(F_6/\Q)$ and order $2$ in the quadratic extension.
The second statement is now easily checked.
\end{proof}

Combining the last two propositions with Lemma \ref{small-counts} and
Corollary \ref{lift-same},
we complete the proof of Theorem \ref{all-counts}.

To close this section, we determine the number of $\F_p$-points on $CY_{13}$
for all $p$ of good reduction.

\begin{prop}\label{point-count}
  For all $p$ different from $2, 3, 13$, the number of
  $\F_p$-points on $CY_{13}$ is equal to $p^3 + 49p^2 + 49p + 1 - a_p$,
  where as before $a_p$ is the eigenvalue of the Hecke operator $T_p$ on
  $f_{13,4}$.
\end{prop}

\begin{proof} (sketch)
  Since $CY_{13}$ is a rigid Calabi-Yau threefold of Euler characteristic $98$
  (Theorem \ref{is-cy}), it has $h^{1,1} = h^{2,2} = 49$.  We have just shown
  that the trace of Frobenius on $H^3$ is $a_p$.  So, by the standard properties
  of \'etale cohomology, it suffices to show that $\Pic CY_{13}$ has a basis
  defined over $\Q$.  Indeed, consider the pullbacks of the $7$ generators of
  $\Pic B$, the $28 + 7 + 3 + 3 = 41$ exceptional divisors, and $K$, a K3
  surface in the fibration of Proposition \ref{k3-fib}.  Let $E_i$ for
  $1 \le i \le 47$ be all of these divisors except for $K$ and $H$,
  the pullback of the hyperplane class from $\P^3$.  The $E_i$ are all
  immobile exceptional divisors, and for each $E_i$ there is a curve that
  meets it and no other generators.

  So it suffices to show that $H, K$ are linearly independent.  
  By the adjunction formula we have $K^2 = 0$ (this holds for any Calabi-Yau
  divisor on a Calabi-Yau variety).  Thus $H\cdot(H \cdot K)$ is twice
  the degree of the image of $K$ in $\P^3$, which is not $0$, while
  $K\cdot(H \cdot K) = H\cdot(K\cdot K) = 0.$
\end{proof}

\section{Modular forms of level $78$ and $390$}
As remarked at (\ref{surf-13}), the double cover of $\P^2$ branched
along $$V_0(-V_0/9-V_1/9+V_2)$$
appears to give a rigid Calabi-Yau threefold that
realizes a twist of the newform $78/4$, where the $V_i$ are defined in 
(\ref{ls4-seqs}).  The proof of this statement is very
similar to that of Theorem \ref{all-counts} and so will not be presented in
detail.  Rather, we will only indicate the changes to the argument for
Theorem \ref{all-counts} that are needed for the proof.

To reduce the level we change the sign.  The expression $V_0/9+V_1/9-V_2$
factors as a product of a linear 
and a cubic polynomial.  The six components of the branch locus intersect in
$23$ curves and all intersections of these are transverse.  All but one of them
are lines.  Also there are $9$ points
where four surfaces meet.  As in Proposition \ref{count-four} we find that
each of these contributes $p^2+2p+1$ points to the double cover and $p^2+p+1$
to the base.

One checks that the intersection of the cubic surface with one of the planes
is a nodal cubic curve.  We call the curve $N$ and the node $p_N$.
All the other components are smooth and meet transversely.

\begin{prop}\label{rigid78}
  The double cover of $\P^3$ branched along $V_0(V_0/9+V_1/9-V_2) = 0$ has a
  rigid Calabi-Yau resolution.
\end{prop}

\begin{proof} As for $CY_{13}$, we construct a Cynk-Szemberg resolution.
  It is unwise to blow up $p_N$, because that would destroy the
  good numerical properties of the double cover; the canonical divisor of 
  $\P^3$ blown up in a point is $-4H+2E$, and we would have a branch locus of
  class $8H-2E$.  We avoid this problem by blowing up the singular curve
  $N$.  This produces a base for the double cover which is nonsingular
  except at one ordinary double point.  We may then take a small resolution
  of this point and continue with the Cynk-Szemberg resolution as before.
  The proof of rigidity is identical to that of Theorem \ref{rigid13}.
\end{proof}

When the blowup of $\P^3$ along $N$ is defined as a map to $\P^3 \times \P^{10}$
given as the identity on $\P^3$ and the basis for the linear system of cubics
vanishing along $N$ used by Magma, the coordinates of the singular point are
$$(-1/3:-1/3:-1/3:1:1/3:2/9:1/9:2/9:1/9:0:1/9:-2/3:-1/3:1).$$
It is easily verified that the tangent cone there is a quadric in $\A^4$ of
determinant $-3$ up to squares and that the branch function takes the value
$5$ there up to squares.  So if $(5/p) = 1$ there are two $\F_p$-rational
double points, and the resolution has rational tangent directions if
and only if $(-3/p) = 1$.  Thus, as in \cite[p.{} 24]{meyer-diss}, taking
a nonprojective small resolution of the nodes adds $2p$ points if 
$(5/p) = (-3/p) = 1$ and subtracts $2p$ points if $(5/p) = 1, (-3/p) = -1$.

Now we need to study the double covers of the intersections of the components
of the branch locus.  We start with $N$, which has $p+1-(-3/p)$ points mod
$p$.  Let $E_0$ be the exceptional divisor for the blowup of $\P^3$ along $N$:
it has $(p+1)(p+1-(-3/p))$ points.
The intersections of the four branch surfaces not containing
$N$ give us six fibres twice each, so they do not contribute to the branch
locus.  The two surfaces that do contain $N$ yield sections.  These surfaces
are disjoint because we blew up their intersection, along which they were
smooth.  So every fibre of $E_0$
is covered by a rational curve, and the number of points is the same on the 
double cover as on $E_0$.

So, by blowing up and resolving the nodes
we removed a curve with $p+1-(-3/p)$ points and
replaced it with a surface with $(p+1)(p+1-(-3/p)) + 2(-3/p)p$ points if
$(5/p) = 1$ or $(p+1)(p+1-(-3/p))$ points if $(5/p) = -1$.  One checks the
four cases to verify that this means we have added $p^2+p+(-15/p)p$ points
over $\F_p$ for any $p>5$.

For the lines, the calculation is much easier.  The exceptional divisors above
a line are all $\P^1 \times \P^1$.  For $7$ of them the branch locus consists
of two sections, and for $3$ of two sections and two fibres, so for these
the double covers all have $p^2+2p+1$ points mod $p$.  The rest consist of 
$3$ with branch curves of classes $F, F, S, S+2F$ and $9$ with $S, S+2F$.
These need to be checked individually, but in no case is it difficult to verify
that the number of points mod $p$ is $p^2+4p+1$.  

We now have enough information to compare the number of points on the octic
to that of the Cynk-Szemberg resolution.
The $9$ fourfold intersections add
$9(p^2+2p)$ to the count.  As in Proposition \ref{meet-point}, the double
covers of exceptional divisors above lines add $10(p^2+p)+12(p^2+3p)-12p$,
since there
are $12$ double points.  Finally, we add $p^2+p+(-15/p)p$ points for the
resolution of the curve $N$, concluding that the resolution of the double
octic has $32p^2+(53+(-15/p))p$ more points than the double octic itself.
This is highly encouraging in light of the following calculation:

\begin{prop}\label{count-78} For all primes $7 \le p \le 200$ the number of
$\F_p$-points on the double cover $t^2 = V_0(V_0/9+V_1/9-V_2)$ is equal to
$p^3+4p^2-17p-(-15/p)p+1-a_p$.  Thus the number of points on the Cynk-Szemberg
resolution is $p^3+36p^2+36p+1-a_p$. \qed
\end{prop}

The proof that the Galois representation for this threefold matches that of
the modular form $78/4$ uses the method of Livn\'e
(\cite[Theorem 1.5]{meyer-diss})
rather than that of Serre-Sch\" utt, since the mod $2$ 
representation is reducible.  By calculation, the only $p < 50$
for which the modular form has $a_p$ odd are $3, 13$.  To prove that there
are no more, we
consider the cubic extensions of $\Q$ unramified outside $2, 3, 13$.  If
the representation were reducible, it would cut out one of these extensions,
and a prime of Frobenius degree $3$ would have odd eigenvalue.  We already
saw in Proposition \ref{unram} that this is not possible, since there is
such a prime $< 50$ in every such extension.  Similarly, for the Galois
representation attached to the Calabi-Yau threefold, we study the
cubic extensions of $\Q$ unramified outside $2, 3, 5, 13$.  From
\cite{tables} we obtain a complete list of these and verify that all $114$
have an inert prime $p \in [17,43]$.

To conclude the proof that the semisimplifications of $\rho_1, \rho_2$ are
isomorphic, let 
$$\Q_S = \Q(\sqrt{-1},\sqrt{2},\sqrt{3},\sqrt{5},\sqrt{13})$$ 
be the compositum of all quadratic extensions of $\Q$ unramified outside
$2,3,5,13$.  Clearly $\Gal (\Q_S/\Q) \cong (\Z/2\Z)^5$.  
It suffices to show that there is a set of primes for which the trace of
Frobenius is the same for $\rho_1, \rho_2$ and for which every cubic
polynomial on $(\Z/2\Z)^5$ vanishing on their Frobenius images in
$\Gal (\Q_S/\Q)$ vanishes on all of $(\Z/2\Z)^5$.  In fact this
holds for the set of primes in $[17,157]$. \qed

%% We expect that the double cover 
%% $t^2 = -\frac{256}{3}V_0 - \frac{5}{6}V_1 + V_2$ will be a rigid Calabi-Yau
%% threefold such that the Galois representation on $H^3$ is the same as that
%% of a quadratic twist of the modular form $390/5$; however, this cannot be 
%% proved, because
%% \cite{tables} does not promise a complete list of extensions of $\Q$
%% unramified outside $2, 3, 5, 13, 23$.  Nevertheless, we have the following
%% result:

\begin{prop}\label{rigid390} The double cover $t^2 = V_0(-512V_0 -5V_1 + 6V_2)$
  has a rigid Calabi-Yau resolution.  If the tables of \cite{tables} are complete
  for cubic extensions of $\Q$ unramified outside $2,3,5,13,23$, it realizes the
modular form $390/5$.
\end{prop}

\begin{remark} 
We do not give a detailed proof because no new techniques are
involved and the desired conclusion remains conditional.
Given a list of extensions known to be complete, it would be
easy to finish the proof that the claimed modular form is realized
(assuming that this is true).
At worst it would be necessary to calculate a few more $a_p$.
\end{remark}

\begin{proof} (sketch)
The branch locus consists of the four coordinate planes and a quartic surface
which is singular at the six points where two of the coordinate 
planes meet the plane $x_0+x_1+x_2+x_3 = 0$ and at $(\alpha:\alpha:\alpha:1)$,
where $\alpha^2 - \alpha - 1/12 = 0$.  Blowing up the six rational singular 
points,
we obtain a branch locus consisting of five components meeting in $22$ curves,
any two of which intersect transversely in a single rational point or not at 
all.  The canonical divisor of $\P^3$ blown up in the six points is 
$-4H+2\sum E_i$, where the $E_i$ are the exceptional divisors.  The class
of the strict transform of the quartic surface is $4H - 2\sum E_i$, because the
singularities are ordinary double points.
The classes
of the components of the branch locus are $H-\sum E_{i,n}$, where $E_{i,n}$ is
the subset of the $E_i$ of divisors above points whose $n$\/th coordinate is 
$0$.  Each $E_i$ appears twice, so the class of the branch locus is $-2$ times
that of the canonical of the base.  Now the surfaces all meet transversely
in rational curves
that meet transversely and we can construct a crepant Cynk-Szemberg resolution.
The proof of rigidity is now identical to that in Theorem \ref{rigid13}.

For counting the points, it is helpful to notice
that $V_0(-512V_0 -5V_1 + 6V_2)$ is invariant under 
permutations of the first three coordinates.
The double covers of the rational singular points have 
$p^2+(2+(-3/p)+(-2/p))p+1$ points mod $p$; by the $\S_3$-invariance it is 
enough to check this for one of them.  They replace $1$ point of $\P^3$.
We then blow up the three fourfold points of intersection, obtaining
$p^2+(1+(-6/p))p+1$ points above each to replace $1$.  The $22$ curves
of intersection fall into $7$ orbits, tabulated below.  Each one is a 
smooth rational curve (indeed, the strict transform of a line in $\P^3$) and
so contains $p+1$ points mod $p$.

\begin{table}[h]
\centering
\begin{tabular}{|l|c|c|}
\hline
Representative&Size of orbit&Number of points\cr\hline
$x=y=0$&$3$&$p^2+2p+1$\cr\hline
$x=w=0$&$3$&$p^2+2p+1$\cr\hline
$x+y=w=0$&$3$&$p^2+(1+2(-2/p))p+1$\cr\hline
$y+z=w=0$&$6$&$p^2+(1+2(-2/p))p+1$\cr\hline
$x+y+w/6=z=0$&$3$&$p^2+(1+2(-2/p))p+1$\cr\hline
$x+y+z=w=0$&$1$&$p^2+2p+1$\cr\hline
$x+y+w=z=0$&$3$&$p^2+2p+1$\cr\hline
\end{tabular}
\smallskip
\caption{Curves in the singular locus of the branch locus of a partial 
resolution of $t^2 = -512V_0-5V_1+6V_2$ and the number of points on the double
covers of their exceptional divisors.}
\end{table}

In addition, we need to resolve the nonrational nodes by small resolutions.
The field of definition of each node is $\Q(\sqrt{3})$, while that of the
lines in an exceptional divisor above a node is 
$F_4 = \Q(\sqrt{3},\sqrt{6\pm 4\sqrt 3})$.  

\begin{defn}
For a prime $p>3$, let 
$\alpha_p = 0$ if
$(3/p) = -1$ or $(-3/p) = -1$.  Otherwise, let $\alpha_p = 2$ if $p$
splits completely in $F_4$ and $-2$ if not.
\end{defn}

If $(3/p) = -1$ the nodes
are not defined over $\F_p$ so the resolution does not affect the count of
$\F_p$-points.  If $(-3/p) = -1$,
then one node contributes $p$ and the other $-p$, so they
cancel out.  If both are $1$, then we get $2p$ if $p$ splits 
completely in $F_4$ and otherwise $-2p$.  Thus by resolving the nodes we
add $\alpha_p p$ points to our total.

Finally, we account for the difference between blowing up curves individually
and in order.  This is not the situation of Proposition 
\ref{meet-point} but that of the following remark, 
because the Picard groups of the exceptional divisors are
not defined over $\Q$.  There are $12$ points where exactly two of
the curves meet, and at each of these it is a double cover 
defined by a function of the form $-2$ times a square that is replaced by an
ordinary curve.  So for each one we subtract $(-2/p)p$ from the count.
In summary, the number of points on the Cynk-Szemberg resolution is greater than
that on the double octic by
$$31p^2+(31+3(-6/p)+6(-3/p)+18(-2/p)+\alpha_p)p.$$
Combining this with a count of points on the double octic, we find that for 
$p \ne 23 \in [17,400]$ the
number of points on the Cynk-Szemberg resolution is 
$$p^3+(32+(-3/p)+3(-2/p))(p^2+p)+1-a_p,$$
where $a_p$ is the Hecke eigenvalue for the newform $390/5$.  This
suggests that there is a basis for the Picard group of the resolution
consisting of $28$ rational divisors, $2$ divisors conjugate over 
$\Q(\sqrt{-3})$, and $3$ pairs of divisors conjugate over $\Q(\sqrt{-2})$.

At this point we need to assume that our lists of fields from \cite{tables}
are complete in order to use Livn\'e's method.  Under this assumption, by 
going up to $97$ we verify that the mod $2$ representations are reducible
and have trace $0$.
As before, let 
$Q_S = \Q(\sqrt{-1},\sqrt{2},\sqrt{3},\sqrt{5},\sqrt{13}, \sqrt{23})$ be the
maximal extension of exponent $2$ unramified outside the primes dividing the
level of the modular form and the primes of bad reduction of the Calabi-Yau.

It suffices to show that there is a set of primes for which the trace of
Frobenius is the same for $\rho_1, \rho_2$ and for which every cubic
polynomial on $(\Z/2\Z)^6$ vanishing on their Frobenius images in
$\Gal (\Q_S/\Q)$ vanishes on all of $(\Z/2\Z)^6$.  In fact this
holds for the set of primes not equal to $23$ and in $[17,353]$.
\end{proof}

\begin{remark} It is striking that for all three of these examples there is
a prime of bad reduction that does not divide the level of the modular form.
This may be
related to the phenomenon described by Sch\" utt in \cite{schutt}.
\end{remark}

\section{Other graphs giving forms of weight $3$ and $4$}

Brown and Schnetz give (\cite[Figure 5]{mf-qft}) graphs for which the point
counts match the modular forms of weight $4$ and level $5, 6, 7, 17$, in 
addition to level $13$ on which we have been concentrating up to now.  They
also give graphs for which the point count matches the forms of weight $3$ and
level $7, 8, 12$.  In this final section we will
apply the methods of this paper to these graphs.
We display graphs isomorphic to those of \cite{mf-qft}
in Figure \ref{graphs-wt-3},
with the order on the vertices given by {\tt genreg} \cite{genreg}.
\begin{defn}\label{genreg-inds} We refer to the $n$\/th graph in the list of
$4$-regular graphs on $k$ vertices produced by {\tt genreg} as $G_{k,n}$.
\end{defn}

\begin{myempty}\label{graphs-wt-3}
\end{myempty}
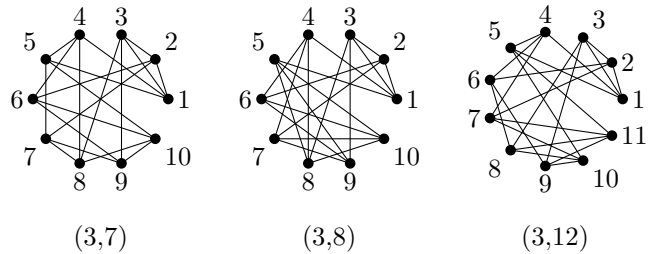
\begin{figure}[ht]
\begin{center}
\begin{tikzpicture}
\matrix[column sep=0.4cm, row sep=0.2cm]{
% (3,7)
\coordinate[label=right:$1$] (u1) at (0:0.9); % {$1$};
\coordinate[label=above right:$2$] (u2) at (2*pi/10 r:0.9); % {$2$};
\coordinate[label=above:$3$] (u3) at (4*pi/10 r:0.9); % {$3$};
\coordinate[label=above:$4$] (u4) at (6*pi/10 r:0.9); % {$4$};
\coordinate[label=above left:$5$] (u5) at (8*pi/10 r:0.9); % {$5$};
\coordinate[label=left:$6$] (u6) at (10*pi/10 r:0.9); % {$6$};
\coordinate[label=below left:$7$] (u7) at (12*pi/10 r:0.9); % {$7$};
\coordinate[label=below:$8$] (u8) at (14*pi/10 r:0.9); % {$8$};
\coordinate[label=below:$9$] (u9) at (16*pi/10 r:0.9); % {$9$};
\coordinate[label=below right:$10$] (u10) at (18*pi/10 r:0.9); % {$10$};
\draw (u1) -- (u2);
\draw (u1) -- (u3);
\draw (u1) -- (u4);
\draw (u1) -- (u5);
\draw (u2) -- (u3);
\draw (u2) -- (u6);
\draw (u2) -- (u7);
\draw (u3) -- (u8);
\draw (u3) -- (u9);
\draw (u4) -- (u5);
\draw (u4) -- (u6);
\draw (u4) -- (u8);
\draw (u5) -- (u7);
\draw (u5) -- (u10);
\draw (u6) -- (u9);
\draw (u6) -- (u10);
\draw (u7) -- (u8);
\draw (u7) -- (u9);
\draw (u8) -- (u10);
\draw (u9) -- (u10);
\fill (u1) circle (2pt) (u2) circle (2pt) (u3) circle (2pt) (u4) circle (2pt) (u5) circle (2pt) (u6) circle (2pt) (u7) circle (2pt) (u8) circle (2pt) (u9) circle (2pt) (u10) circle (2pt);&
% (3,8)
\coordinate[label=right:$1$] (v1) at (0:0.9); % {$1$};
\coordinate[label=above right:$2$] (v2) at (2*pi/10 r:0.9); % {$2$};
\coordinate[label=above:$3$] (v3) at (4*pi/10 r:0.9); % {$3$};
\coordinate[label=above:$4$] (v4) at (6*pi/10 r:0.9); % {$4$};
\coordinate[label=above left:$5$] (v5) at (8*pi/10 r:0.9); % {$5$};
\coordinate[label=left:$6$] (v6) at (10*pi/10 r:0.9); % {$6$};
\coordinate[label=below left:$7$] (v7) at (12*pi/10 r:0.9); % {$7$};
\coordinate[label=below:$8$] (v8) at (14*pi/10 r:0.9); % {$8$};
\coordinate[label=below:$9$] (v9) at (16*pi/10 r:0.9); % {$9$};
\coordinate[label=below right:$10$] (v10) at (18*pi/10 r:0.9); % {$10$};
\draw (v1) -- (v2);
\draw (v1) -- (v3);
\draw (v1) -- (v4);
\draw (v1) -- (v5);
\draw (v2) -- (v3);
\draw (v2) -- (v6);
\draw (v2) -- (v7);
\draw (v3) -- (v8);
\draw (v3) -- (v9);
\draw (v4) -- (v6);
\draw (v4) -- (v7);
\draw (v4) -- (v8);
\draw (v5) -- (v8);
\draw (v5) -- (v9);
\draw (v5) -- (v10);
\draw (v6) -- (v9);
\draw (v6) -- (v10);
\draw (v7) -- (v9);
\draw (v7) -- (v10);
\draw (v8) -- (v10);
\fill (v1) circle (2pt) (v2) circle (2pt) (v3) circle (2pt) (v4) circle (2pt) (v5) circle (2pt) (v6) circle (2pt) (v7) circle (2pt) (v8) circle (2pt) (v9) circle (2pt) (v10) circle (2pt); &

%(3,12)
\coordinate[label=right:$1$] (w1) at (0:0.9); % {$1$};
\coordinate[label=right:$2$] (w2) at (2*pi/11 r:0.9); % {$2$};
\coordinate[label=above right:$3$] (w3) at (4*pi/11 r:0.9); % {$3$};
\coordinate[label=above:$4$] (w4) at (6*pi/11 r:0.9); % {$4$};
\coordinate[label=above left:$5$] (w5) at (8*pi/11 r:0.9); % {$5$};
\coordinate[label=left:$6$] (w6) at (10*pi/11 r:0.9); % {$6$};
\coordinate[label=left:$7$] (w7) at (12*pi/11 r:0.9); % {$7$};
\coordinate[label=below left:$8$] (w8) at (14*pi/11 r:0.9); % {$8$};
\coordinate[label=below:$9$] (w9) at (16*pi/11 r:0.9); % {$9$};
\coordinate[label=below right:$10$] (w10) at (18*pi/11 r:0.9); % {$10$};
\coordinate[label=right:$11$] (w11) at (20*pi/11 r:0.9); % {$10$}; 
\draw (w1) -- (w2);
\draw (w1) -- (w3);
\draw (w1) -- (w4);
\draw (w1) -- (w5);
\draw (w2) -- (w3);
\draw (w2) -- (w6);
\draw (w2) -- (w7);
\draw (w3) -- (w8);
\draw (w3) -- (w9);
\draw (w4) -- (w5);
\draw (w4) -- (w6);
\draw (w4) -- (w7);
\draw (w5) -- (w10);
\draw (w5) -- (w11);
\draw (w6) -- (w8);
\draw (w6) -- (w9);
\draw (w7) -- (w10);
\draw (w7) -- (w11);
\draw (w8) -- (w10);
\draw (w8) -- (w11);
\draw (w9) -- (w10);
\draw (w9) -- (w11);
\fill (w1) circle (2pt) (w2) circle (2pt) (w3) circle (2pt) (w4) circle (2pt) (w5) circle (2pt) (w6) circle (2pt) (w7) circle (2pt) (w8) circle (2pt) (w9) circle (2pt) (w10) circle (2pt) (w11) circle (2pt); \\
\node{(3,7)}; & \node{(3,8)}; & \node{(3,12)}; \\}; 
%% \node[below=2cm,align=flush center,text width=10cm]
%% {{\sc Figure 2.} Graphs for which the $c_p$ are congruent mod $p$ to the eigenvalues of forms of weight $3$ and level $7, 8, 12$. };

\end{tikzpicture}
\end{center}
\caption{Graphs for which the $c_p$ are congruent mod $p$ to the eigenvalues
of forms of weight $3$ and level $7, 8, 12$.}\label{graphs-wt-3}
\end{figure}

\subsection{Weight 3, level 7}\label{w3-lev7}
There is very little to add to the discussion of this example in \cite{k3-phi4}.
The graph is isomorphic to $G_{10,54}$.
As in \cite{k3-phi4}, we choose to delete vertex $3$ and begin by computing
the same five-invariant, which in our numbering comes from the edges
$$(1,2),(1,4),(1,5),(2,7),(4,5).$$
Instead of reducing the sequence of edges of \cite{k3-phi4},
we reduce the sequence
$$(2,6),(4,8),(6,9),(6,10),(7,9),(9,10),$$ obtaining a polynomial in $5$
variables to which subspace reduction for $x_0 = x_1 = x_2$ applies.
The Segre embedding produces a surface of degree $8$ in $\P^5$ with
canonical singularities, one each of types $A_4, A_3, A_2$ and four $A_1$, and
trivial canonical divisor.  One can proceed as in \cite{k3-phi4} or as in the
next example to show that the
cusp form of weight $3$ associated to the surface is of level $7$.

\subsection{Weight 3, level 8}\label{w3-lev8}
We begin by deleting vertex $2$.  Lemma 55 of \cite{k3-phi4} does not apply
here, but by computing the five-invariant for the edges
$$(1,4),(3,8),(4,7),(6,10),(7,10)$$
and the linear and resultant reductions for edges
$$(7,9),(1,3),(1,5),(3,9),(5,8).$$
we find a polynomial in $6$ variables on which normal reduction can be
used to obtain a double cover of $\P^3$ branched along an octic, which is
$$ \begin{aligned}
&(x_1+x_2)(x_0^2x_2 + x_0x_1x_2 + x_1x_2x_3 + x_1x_3^2) \times \cr
&\quad x_0^2x_1x_2 + x_0x_1^2x_2 + x_0^2x_2^2 + x_0x_1x_2^2 + 
        4x_0x_1x_2x_3 \cr
&\quad + x_1^2x_2x_3 + x_1x_2^2x_3 + x_1^2x_3^2 + 
        x_1x_2x_3^2. \cr
\end{aligned}$$
Now, this octic is quite unlike the ones discussed in previous sections:
the quartic component is singular along the line $x_1 = x_2 = 0$, which is
contained in the other two components.  Likewise, if we convert the double
cover to a quintic $Q$ by means of Proposition \ref{to-hyp}, we find that every
monomial in the equation defining $Q$ vanishes to order $3$ along the line
$L: x_0 = x_2 = x_3 = 0$.  So we apply Proposition \ref{lin-sim} to obtain a
prime-similar subvariety of $\P^1 \times \P^2$.
%% has a singularity along the line $L: x_0 = x_2 = x_3 = 0$ which is not like the
%% singular curves on quintics previously discussed.  The obvious thing to do
%% is to blow up the quintic along $L$ to obtain a subvariety
%% $V_{4,2} \subset \P^4 \times \P^2$.

%% \begin{prop}\label{blowup-line}
%% $V_{4,2}$ is prime-similar to a point.
%% \end{prop}

%% \begin{proof}
%% Consider the projection to the second factor over $\F_p$.  
%% The inverse image of a generic point
%% is a smooth conic, which has $p+1$ points.  The only possible degenerations 
%% are the union of two coplanar lines (rational or not) or a double line,
%% and for all of these the number of points is $1 \bmod p$.  So we only need
%% to worry about fibres of the projection that have dimension $2$.  However,
%% any such fibre is defined inside $Q$ by two linear equations and is therefore
%% contained in the plane defined by the same equations.  So if it has dimension
%% $2$ it is a plane, and again the number of points is $1 \bmod p$.

%% There are $p^2+p+1$ fibres, and on each the number of points is $1 \bmod p$,
%% so the total number of points is $1 \bmod p$ for all $p$.
%% \end{proof}

%% The blowup $Q \to V_{4,2}$ is an isomorphism away from $L$, so as in 
%% Proposition \ref{red-simil} we conclude that $Q$ is prime-similar to the
%% exceptional divisor $E$.  Since $x_0 = x_2 = x_3$ on $E$ it is naturally
%% embedded in $\P^1 \times \P^2$.  

The Segre embedding makes it a subvariety of
$\P^5$ defined by the equations
$$\begin{aligned}
&x_1x_5 - x_2x_4, \quad x_0x_5 - x_2x_3, \quad x_0x_4 - x_1x_3, \cr
&x_0x_2x_3 - x_1x_2x_4 - 2x_2^2x_4 - x_2^2x_5 - 4x_2x_4^2 - 
    4x_2x_4x_5 + x_3^2x_4 - x_4^3 - 2x_4^2x_5 - x_4x_5^2, \cr
&\quad x_0^2x_2 - x_1^2x_2 - 2x_1x_2^2 - 4x_1x_2x_4 + x_1x_3^2 - x_1x_4^2 - x_2^3 - 4x_2^2x_4 - 2x_2x_4^2 - x_2x_4x_5. \cr
\end{aligned}
$$
Projecting away from $(0:0:0:-1:0:1)$ gives a birational
equivalence with a surface in $\P^3$ defined by a quadric and a cubic with
only canonical singularities, so this is a K3 surface.
It has seven singular points, four $A_2$ and three $A_1$,
and there are $14$ lines through one or more of the singularities.
It is routine to work out the intersection matrix of the $11$ exceptional
curves and these lines.
The matrix has rank $20$, which proves
that the graph is modular of weight $3$, and when five rows and columns
corresponding to generators of the kernel are deleted the determinant is $8$,
so the modular form is a twist of the one of level $8$.

To prove that the modular form is indeed the newform of level $8$, we 
pass to $P$, the projection away from $(0:0:0:-1:0:1)$ described earlier, which
is clearly prime-similar.  We 
verify, using Lemma \ref{bad-primes}, that $P$ has good reduction outside
$\{2,3\}$.  Also, it has an $A_3$ singularity, three $A_2$, and three $A_1$,
so the number of $\F_p$-points on a resolution is $12p$ more than on $P$.
The trace of Frobenius acting on the two-dimensional component of
$H^2(\tilde P)$ is therefore $[P]_p - p^2 - 8p - 1$.

Next we check that the trace of
Frobenius at $p$ is even for a set of primes of good reduction
and including one 
prime inert in every cubic extension of $\Q$ unramified outside $2, 3$.
From \cite{tables} there are $9$ such fields and we need to go up to
$p = 19$ to exclude them all.  To apply Livn\'e's method, we need to find
primes of good reduction with all possible Frobenius elements in 
$\Gal(\Q(\sqrt{-1},\sqrt{2},\sqrt{3})/\Q)$ for which $a_p = \tr \Frob_p$;
it suffices to consider the primes up to $23$.  This confirms that
the modular form is that of level $8$.

\subsection{Weight 3, level 12}
We start by reducing the projectivized
graph hypersurface to a hypersurface of degree $6$
in $\P^5$ by deleting vertex $3$, computing the five-invariant for the edges
$$(1,2),(1,4),(4,5),(7,11),(8,11)$$
and then reducing the edges
$$(1,5),(2,6),(2,7),(4,6),(6,8),(6,9),(8,10).$$
Once that is done, we apply subspace reduction to the line 
$x_0 = x_1 = x_2 = x_3 = 0$ and then Proposition \ref{to-hyp} to obtain a
quintic $Q_{12}$ in $\P^4$ defined by
$$\begin{aligned}
&x_0^2x_1x_2^2 + x_0^2x_1x_2x_3 + x_0^2x_1x_2x_4 + x_0^2x_1x_3x_4 + 
    x_0^2x_2^2x_3 \cr
&\quad + 2x_0^2x_2x_3x_4 + x_0^2x_3x_4^2 - 4x_1^2x_2^2x_4 
    - 8x_1^2x_2x_3x_4 - 4x_1^2x_3^2x_4\cr
&\quad - 4x_1x_2x_3x_4^2 - 4x_1x_3^2x_4^2 + 4x_2x_3^2x_4^2.\cr
\end{aligned}$$

None of the reduction techniques of Section \ref{sec:red} applies to this 
quintic in $\P^4$, nor to any other that can be constructed from the 
graph by the methods of this paper.
Let us consider the linear system of quadrics vanishing along
the three lines in the singular subscheme of $Q_{12}$ that contain the point
$(0:1:0:0:0)$ and at the isolated points of the singular subscheme: a basis is
$$x_0x_2, x_0x_3, x_0x_4, x_1x_2+x_1x_3,x_3x_4.$$
It defines an invertible rational map from $Q_{12}$ to a threefold $T$ which
is of degree $4$ and therefore prime-trivial.  The components of the loci along 
which this map
and its inverse are not defined are all shown to be prime-trivial by the
Chevalley-Warning theorem, except that $(0:1:0:0:0) \in Q_{12}$ goes to the
surface in $\P^4$ defined by
$$\begin{aligned}
& x_0 + x_1 = x_1^3x_4 - x_1^2x_2x_3 - 2x_1^2x_2x_4 + x_1x_2^2x_3 \cr
&\quad + x_1x_2^2x_4 - 4x_1x_4^3 - 4x_2x_3^2x_4 - 4x_2x_3x_4^2 = 0.\cr
\end{aligned}
$$
which is therefore prime-similar to $Q_{12}$.  Removing the linear equation,
we get a surface $S_3$ in $\P^3$, which has singularities of type $A_6, A_4, A_2$
at $(0:1:0:0),(0:0:1:0),(1:1:0:0)$.

\begin{remark} This construction could be described in terms of blowing up
$Q_{12}$ at $(0:1:0:0:0)$.  The exceptional divisor consists of three planes
and the blowup is singular along one of them.  If we normalize the blowup
along this plane and pull back the plane to the normalization, we obtain a
surface birationally equivalent to $S_3$.  But in the absence of a result 
comparable to Proposition \ref{lin-sim}
that shows how to choose the right blowup a priori,
this is only a trick, not a method.
\end{remark}

The quadrics vanishing at the first two
of these give a birational equivalence with a surface $S_7 \subset \P^7$ with an
$A_4$ singularity at $(0:0:0:-1:0:1:0:0)$, and projecting away from this point
we obtain a model $S_6 \subset \P^6$ with two $A_1$ and four $A_2$ 
singularities. 
Applying Lemma \ref{bad-primes} to the the singular subscheme of $S_3$ we find 
that $S_3$ has good reduction outside $\{2,3,5,7\}$, and in fact $S_6$ has
good reduction at $7$.

The sublattice of $\Pic S_6$ generated by the components of exceptional 
divisors, lines passing through a singular point, and lines introduced by 
blowing up the $A_4$ singularity has rank $19$ and is fixed by 
$\Gal(\bar \Q/\Q)$.  To obtain a twentieth generator, we take the line
$x_0 - x_1 = x_2 - \alpha x_3$ on $S_3$, where $\alpha^2+\alpha+1 = 0$.
It is not fixed by the action of $\Gal (\bar \Q/\Q)$,
so it is independent of the previous ones, and
$\Pic S_3$ has rank $20$.

All of the components of the irreducible divisors of a resolution $\tilde S_6$
of $S_6$
are rational, so the resolution has $10p$ more $\F_p$-points than $S_6$ does.
So the trace of the representation on $H^{2,0}+H^{0,2}$ of $\tilde S_6$ at
$\Frob_p$ is
$[S_6]_p+10p-p^2-19p-(-3/p)-1$.  To apply Livn\'e's method, we first verify
that the trace is always even.  There are $32$ cubic fields unramified outside
$\{2,3,5\}$ \cite{tables}, but every one of them has an inert prime $\le 19$,
so it is enough to check that far.  To conclude, we show that the trace
is equal to $a_p$ for a set of primes $> 5$ whose Frobenius 
elements give a non-cubic subset in $\Q(\sqrt{-1},\sqrt 2, \sqrt 3, \sqrt 5)$.
We only need to go up to $p = 73$.

%\begin{myempty}\label{graphs-wt-4}
%\end{myempty}
%\includegraphics{rev-rigid-cy3-figure2.pdf}
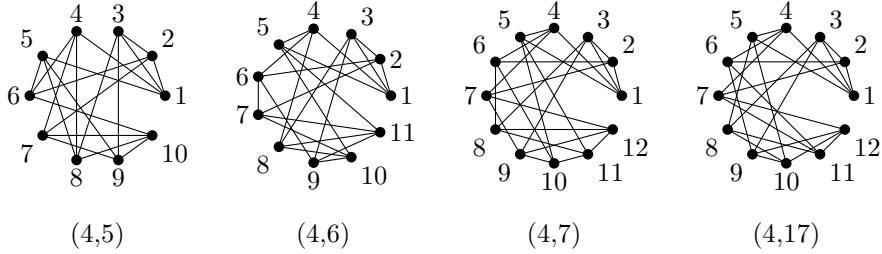
\begin{figure}[ht]\label{graphs-wt-4}
\begin{center}
\begin{tikzpicture}
%\matrix[column sep=0.4cm, row sep=0.2cm]{
\path node[matrix,column sep=0.4cm, row sep=0.2cm]{
%% (4,5)
\coordinate[label=right:$1$] (v1) at (0:0.9); % {$1$};
\coordinate[label=above right:$2$] (v2) at (2*pi/10 r:0.9); % {$2$};
\coordinate[label=above:$3$] (v3) at (4*pi/10 r:0.9); % {$3$};
\coordinate[label=above:$4$] (v4) at (6*pi/10 r:0.9); % {$4$};
\coordinate[label=above left:$5$] (v5) at (8*pi/10 r:0.9); % {$5$};
\coordinate[label=left:$6$] (v6) at (10*pi/10 r:0.9); % {$6$};
\coordinate[label=below left:$7$] (v7) at (12*pi/10 r:0.9); % {$7$};
\coordinate[label=below:$8$] (v8) at (14*pi/10 r:0.9); % {$8$};
\coordinate[label=below:$9$] (v9) at (16*pi/10 r:0.9); % {$9$};
\coordinate[label=below right:$10$] (v10) at (18*pi/10 r:0.9); % {$10$};
%\coordinate[label=below:$(4,5)$] (vlab) at (3*pi/2 r:0.9.8);
\draw (v1) -- (v2);
\draw (v1) -- (v3);
\draw (v1) -- (v4);
\draw (v1) -- (v5);
\draw (v2) -- (v3);
\draw (v2) -- (v6);
\draw (v2) -- (v7);
\draw (v3) -- (v8);
\draw (v3) -- (v9);
\draw (v4) -- (v6);
\draw (v4) -- (v7);
\draw (v4) -- (v8);
\draw (v5) -- (v6);
\draw (v5) -- (v8);
\draw (v5) -- (v9);
\draw (v6) -- (v10);
\draw (v7) -- (v9);
\draw (v7) -- (v10);
\draw (v8) -- (v10);
\draw (v9) -- (v10);
\fill (v1) circle (2pt) (v2) circle (2pt) (v3) circle (2pt) (v4) circle (2pt) (v5) circle (2pt) (v6) circle (2pt) (v7) circle (2pt) (v8) circle (2pt) (v9) circle (2pt) (v10) circle (2pt); & % (v11) circle (2pt); &

%(4,6)
\coordinate[label=right:$1$] (w1) at (0:0.9); % {$1$};
\coordinate[label=right:$2$] (w2) at (2*pi/11 r:0.9); % {$2$};
\coordinate[label=above right:$3$] (w3) at (4*pi/11 r:0.9); % {$3$};
\coordinate[label=above:$4$] (w4) at (6*pi/11 r:0.9); % {$4$};
\coordinate[label=above left:$5$] (w5) at (8*pi/11 r:0.9); % {$5$};
\coordinate[label=left:$6$] (w6) at (10*pi/11 r:0.9); % {$6$};
\coordinate[label=left:$7$] (w7) at (12*pi/11 r:0.9); % {$7$};
\coordinate[label=below left:$8$] (w8) at (14*pi/11 r:0.9); % {$8$};
\coordinate[label=below:$9$] (w9) at (16*pi/11 r:0.9); % {$9$};
\coordinate[label=below right:$10$] (w10) at (18*pi/11 r:0.9); % {$10$};
\coordinate[label=right:$11$] (w11) at (20*pi/11 r:0.9); % {$10$}; 
\draw (w1) -- (w2);
\draw (w1) -- (w3);
\draw (w1) -- (w4);
\draw (w1) -- (w5);
\draw (w2) -- (w3);
\draw (w2) -- (w6);
\draw (w2) -- (w7);
\draw (w3) -- (w8);
\draw (w3) -- (w9);
\draw (w4) -- (w5);
\draw (w4) -- (w6);
\draw (w4) -- (w8);
\draw (w5) -- (w10);
\draw (w5) -- (w11);
\draw (w6) -- (w7);
\draw (w6) -- (w9);
\draw (w7) -- (w10);
\draw (w7) -- (w11);
\draw (w8) -- (w10);
\draw (w8) -- (w11);
\draw (w9) -- (w10);
\draw (w9) -- (w11);
\fill (w1) circle (2pt) (w2) circle (2pt) (w3) circle (2pt) (w4) circle (2pt) (w5) circle (2pt) (w6) circle (2pt) (w7) circle (2pt) (w8) circle (2pt) (w9) circle (2pt) (w10) circle (2pt) (w11) circle (2pt); &

%(4,7)
\coordinate[label=right:$1$] (x1) at (0:0.9); % {$1$};
\coordinate[label=above right:$2$] (x2) at (2*pi/12 r:0.9); % {$2$};
\coordinate[label=above right:$3$] (x3) at (4*pi/12 r:0.9); % {$3$};
\coordinate[label=above:$4$] (x4) at (6*pi/12 r:0.9); % {$4$};
\coordinate[label=above left:$5$] (x5) at (8*pi/12 r:0.9); % {$5$};
\coordinate[label=above left:$6$] (x6) at (10*pi/12 r:0.9); % {$6$};
\coordinate[label=left:$7$] (x7) at (12*pi/12 r:0.9); % {$7$};
\coordinate[label=below left:$8$] (x8) at (14*pi/12 r:0.9); % {$8$};
\coordinate[label=below left:$9$] (x9) at (16*pi/12 r:0.9); % {$9$};
\coordinate[label=below:$10$] (x10) at (18*pi/12 r:0.9); % {$10$};
\coordinate[label=below right:$11$] (x11) at (20*pi/12 r:0.9); % {$10$}; 
\coordinate[label=below right:$12$] (x12) at (22*pi/12 r:0.9); % {$11$}; 
\draw (x1) -- (x2);
\draw (x1) -- (x3);
\draw (x1) -- (x4);
\draw (x1) -- (x5);
\draw (x2) -- (x3);
\draw (x2) -- (x6);
\draw (x2) -- (x7);
\draw (x3) -- (x8);
\draw (x3) -- (x9);
\draw (x4) -- (x5);
\draw (x4) -- (x6);
\draw (x4) -- (x7);
\draw (x5) -- (x10);
\draw (x5) -- (x11);
\draw (x6) -- (x8);
\draw (x6) -- (x10);
\draw (x7) -- (x9);
\draw (x7) -- (x12);
\draw (x8) -- (x11);
\draw (x8) -- (x12);
\draw (x9) -- (x10);
\draw (x9) -- (x12);
\draw (x10) -- (x11);
\draw (x11) -- (x12);
\fill (x1) circle (2pt) (x2) circle (2pt) (x3) circle (2pt) (x4) circle (2pt) (x5) circle (2pt) (x6) circle (2pt) (x7) circle (2pt) (x8) circle (2pt) (x9) circle (2pt) (x10) circle (2pt) (x11) circle (2pt) (x12) circle (2pt); &

%(4,17)
\coordinate[label=right:$1$] (y1) at (0:0.9); % {$1$};
\coordinate[label=above right:$2$] (y2) at (2*pi/12 r:0.9); % {$2$};
\coordinate[label=above right:$3$] (y3) at (4*pi/12 r:0.9); % {$3$};
\coordinate[label=above:$4$] (y4) at (6*pi/12 r:0.9); % {$4$};
\coordinate[label=above left:$5$] (y5) at (8*pi/12 r:0.9); % {$5$};
\coordinate[label=above left:$6$] (y6) at (10*pi/12 r:0.9); % {$6$};
\coordinate[label=left:$7$] (y7) at (12*pi/12 r:0.9); % {$7$};
\coordinate[label=below left:$8$] (y8) at (14*pi/12 r:0.9); % {$8$};
\coordinate[label=below left:$9$] (y9) at (16*pi/12 r:0.9); % {$9$};
\coordinate[label=below:$10$] (y10) at (18*pi/12 r:0.9); % {$10$};
\coordinate[label=below right:$11$] (y11) at (20*pi/12 r:0.9); % {$10$}; 
\coordinate[label=below right:$12$] (y12) at (22*pi/12 r:0.9); % {$11$}; 
\draw (y1) -- (y2);
\draw (y1) -- (y3);
\draw (y1) -- (y4);
\draw (y1) -- (y5);
\draw (y2) -- (y3);
\draw (y2) -- (y6);
\draw (y2) -- (y7);
\draw (y3) -- (y8);
\draw (y3) -- (y9);
\draw (y4) -- (y5);
\draw (y4) -- (y6);
\draw (y4) -- (y7);
\draw (y5) -- (y8);
\draw (y5) -- (y10);
\draw (y6) -- (y9);
\draw (y6) -- (y11);
\draw (y7) -- (y11);
\draw (y7) -- (y12);
\draw (y8) -- (y10);
\draw (y8) -- (y11);
\draw (y9) -- (y10);
\draw (y9) -- (y12);
\draw (y10) -- (y12);
\draw (y11) -- (y12);
\fill (y1) circle (2pt) (y2) circle (2pt) (y3) circle (2pt) (y4) circle (2pt) (y5) circle (2pt) (y6) circle (2pt) (y7) circle (2pt) (y8) circle (2pt) (y9) circle (2pt) (y10) circle (2pt) (y11) circle (2pt) (y12) circle (2pt); \\
\node{(4,5)}; & \node{(4,6)}; & \node{(4,7)}; & \node{(4,17)};\\};
%% \node[below=2cm,align=flush center,text width=10cm]
%% {{\sc Figure 3.} Graphs for which the $c_p$ are congruent mod $p$ to the eigenvalues
%% of forms of weight $4$ and level $5, 6, 7, 17$.};
\end{tikzpicture}
\end{center}
\caption{Graphs for which the $c_p$ are congruent mod $p$ to the eigenvalues
of forms of weight $4$ and level $5, 6, 7, 17$.}
\end{figure}

\subsection{Weight 4, level 5}\label{w4-lev5}
We now turn to graphs that give forms of weight $4$.
The graph corresponding to the form of level $5$ is $G_{10,56}$.
Lemma 55 of \cite{k3-phi4} does not apply to any of the graphs obtained
by deleting one vertex from this graph, but we can still attempt the reduction.
Using linear, resultant, and normal reduction has not yet led to 
varieties of dimension less than $4$.  For example, if
we delete vertex $2$, take the five-invariant for edges 
$$(1,3),(1,4),(1,5),(3,8),(7,9),$$ and then reduce by edges
$$(3,9),(4,6),(9,10),(4,7),(7,10),$$ we obtain a fourfold whose point counts
match the newform of weight $4$ and level $5$ for small $p$, as claimed in
\cite{mf-qft}.  
Fortunately, Proposition \ref{lin-sim} comes to our rescue.  It applies to
the set of variables $\{x_0,x_2,x_3,x_5\}$ to give us a double cover $O_5$
of an octic in $\P^3$, defined by
$$\begin{aligned}
t^2 = & x_2(x_1+x_2+x_3)(x_0x_1+x_0x_3+x_1x_3)(x_0x_1^2x_2 + x_0x_1x_2^2 + 4x_0^2x_1x_3 + 4x_0x_1^2x_3 \cr
&\quad + 4x_0^2x_2x_3 + 6x_0x_1x_2x_3 + x_1^2x_2x_3 + x_0x_2^2x_3 + 
x_1x_2^2x_3 + x_0x_2x_3^2 + x_1x_2x_3^2).
\end{aligned}$$
The following conjecture has been checked for $p < 200$.

\begin{conj}\label{w4l5}
For all primes $p>2$
the threefold $O_5$ has $p^3+5p^2-(9+(-1/p))p+1-a_p$
points over $\F_p$, where $a_p$ is the
eigenvalue for $T_p$ of the newform of weight $4$ and level $5$.
\end{conj}

Most likely this conjecture
can be proved by constructing a crepant resolution with
$cp^2+(c+14+(-1/p))p$ more $\F_p$-points than $O_5$ for all $p$, for some
suitable $c$, and showing that it is a rigid Calabi-Yau threefold.

Several examples of rigid Calabi-Yau threefolds realizing this newform are
described in \cite[Section 6.1.1]{meyer-diss}.  The Tate conjecture
predicts that there are correspondences between these and the one described
here; since the construction is apparently unrelated, it would be very 
interesting to find such a correspondence.

\subsection{Weight $4$, level $6$}\label{w4-lev6}
Similarly, for level $6$, we start from $G_{11,241}$.  We delete vertex $5$,
take the five-invariant for the edges $$(1,2),(1,3),(2,3),(1,4),(4,8)$$
(to which \cite[Lemma 55]{k3-phi4} applies), then reduce edges in the order
$$(4,6),(2,6),(7,10),(7,11),(8,10),(9,10)$$ to obtain a fivefold for which
the number of points mod $p$ matches the newform of weight $4$ and level $6$.

First we apply subspace reduction to the variables $\{x_0, x_3\}$
to produce a hypersurface
in $\P^4 \times \P^1$ of bidegree $(5,2)$, which is birational to a double cover
of $\P^4$ branched over a threefold of degree $10$.  The branch locus has 
components of degree $4, 6$, so by Proposition \ref{to-hyp} we return to a
sextic hypersurface in $\P^5$.  This one satisfies the hypothesis of
Proposition \ref{lin-sim} for $\{x_0, x_1, x_2, x_3\}$, so we may apply subspace
reduction again.  We find a double octic 
$O_6$ defined by
$$\begin{aligned}
t^2 &= -x_1x_2x_3(x_0^4x_2 - 2x_0^2x_1^2x_2 + x_1^4x_2 - 4x_0^2x_1x_2^2 + 
4x_1^3x_2^2 - 2x_0^2x_2^3 + 6x_1^2x_2^3 \cr
&\quad + 4x_1x_2^4 + x_2^5 + x_0^4x_3 - 2x_0^2x_1^2x_3 + x_1^4x_3 - 12x_0^2x_1x_2x_3 \cr
&\quad + 8x_1^3x_2x_3 - 6x_0^2x_2^2x_3 + 18x_1^2x_2^2x_3 + 16x_1x_2^3x_3 
+ 5x_2^4x_3 - 4x_0^2x_1x_3^2 \cr
&\quad + 4x_1^3x_3^2 - 4x_0^2x_2x_3^2 + 12x_1^2x_2x_3^2 + 12x_1x_2^2x_3^2 + 
4x_2^3x_3^2).\cr
\end{aligned}$$

Let $\alpha(p) = 1$ if $p \equiv 1 \bmod 8$ and $0$ otherwise.  We then have the
following conjecture:

\begin{conj}\label{w4l6}
For all primes $p>3$ the number of $\F_p$-points of $O_6$ is equal to
$p^3+5p^2-(11+4\alpha(p)+(-3/p))p+1-a_p$, where $a_p$ is the eigenvalue of $T_p$
for the newform of weight $4$ and level $6$.
\end{conj}

This conjecture has been checked for $p < 200$, and much
the same remarks on a proof apply to it as to Conjecture \ref{w4l5}.
Again in \cite[Section 6.1.2]{meyer-diss}, we find many rigid threefolds
realizing the same form, and it would be interesting to find a correspondence.
The term $4\alpha(p)$ is the number of primes of degree
$1$ above $p$ in $\Q(\zeta_8)$; this is
explained by the quintic component of the branch locus having singularities
at $(-2\zeta_8:\zeta_8^2-1:-\zeta_8^2-1:1)$ and its conjugates.

\begin{remark} It is possible to reach dimension $4$ by only linear and 
resultant reduction, but subspace reduction does not apply to these varieties.
\end{remark}

\subsection{Weight $4$, level $7$}
The graph is $G_{12,1330}$.
If we delete vertex $5$, take the $5$-invariant for the edges
$$(1,2),(1,3),(2,3),(1,4),(4,7),$$ and then reduce the edges in the order
$$(4,6),(2,6),(8,12),(11,12),(6,10),(10,11),(8,11),(9,10),(7,9),$$
we obtain a subvariety of $\P^5$ to which the normal reduction can be applied.
Converting the resulting double cover of $\P^3$ into a quintic by means of
Proposition \ref{to-hyp}, we then consider the linear system of quadrics
vanishing on the singular points contained in components of the singular
subscheme of degree greater than $8$.  It gives us a new quintic.  Again
we take quadrics vanishing on singular points in components of the singular
subscheme of degree $8$; this gives us a map to $\P^6$.  The singular 
subscheme of the image has one component $C_1$ of degree $7$ and three 
$D_1, D_2, D_3$ of degree $19$;
if we project away from the point in $C_1$ and the point in the correct one of
the $D_i$ we obtain a quintic that can be converted back into a double cover
$O_7$ of $\P^3$ with the equation
$$\begin{aligned}
t^2 &= x_0(x_0x_1x_2 + x_1^2x_2 + x_1^2x_3 - x_0x_2x_3 - x_1x_2x_3 - x_2^2x_3 - x_1x_3^2 - x_2x_3^2) \times\cr
&\quad (x_0^2x_1x_2 + x_0x_1^2x_2 + x_0x_1^2x_3 - x_0^2x_2x_3 - x_0x_1x_2x_3 \cr
&\quad + 4x_1^2x_2x_3 - x_0x_2^2x_3 - 4x_1x_2^2x_3 - x_0x_1x_3^2 - x_0x_2x_3^2 - 4x_1x_2x_3^2).\cr
\end{aligned}$$

We find the formula
$$p^3+5p^2-(11+3(-3/p)+(5/p))p+1-a_p$$
for the number of $\F_p$-points of $O_7$  for $5 \le p \le 200$,
where $a_p$ is the Hecke eigenvalue for the newform of weight $4$ and level 
$7$.  As in the last two examples we conjecture this to be true for all 
larger $p$.

It seems that the only Calabi-Yau threefold known to be associated to this
modular form is that provided by the Kuga-Sato construction, namely a resolution
of the fibre square of the universal elliptic curve over $X_0(7)$
\cite[Section 1]{deligne}.

\subsection{Weight $4$, level $17$}
The graph here is $G_{12,1321}$.
If we delete vertex $5$, take the $5$-invariant with respect to the edges
$$(1,2),(1,3),(2,3),(1,4),(4,7),$$ 
and perform linear and resultant reductions for the edges
$$(4,6),(2,6),(8,11),(9,10),(9,12),(8,10),(3,8),(10,12),(7,12),$$
we obtain a polynomial of degree $6$ in $6$ variables to which the normal
reduction can be applied.  Thus we obtain a double cover of $\P^3$ branched
along surfaces of degree $3,5$, and we can convert this to a quintic as in
Proposition \ref{to-hyp}.  We simplify this quintic by considering polynomials
of degree $2$ vanishing on the points of the singular subscheme contained
in components of degree $\ge 8$.  This gives a subscheme of $\P^6$ of degree
$11$, and projecting away from a singular point of degree $13$ and one of the
two of degree $23$ produces a quintic that can be mapped back to a double cover
of $\P^3$ with branch components of degree $1,3,4$.  The associated double octic
$O$ is defined by
$$\begin{aligned}
&t^2 - x_0(x_1^2x_2 + x_0x_1x_3 + x_1^2x_3 - x_1x_2x_3 - x_2^2x_3 - x_1x_3^2 - x_2x_3^2) \times \cr
&\quad (x_0x_1^2x_2 + x_0^2x_1x_3 + x_0x_1^2x_3 + 3x_0x_1x_2x_3 + 4x_1^2x_2x_3 \cr
&\quad - x_0x_2^2x_3 - 4x_1x_2^2x_3 - x_0x_1x_3^2 - x_0x_2x_3^2 - 4x_1x_2x_3^2) = 0.\cr
\end{aligned}
$$

As found in \cite{mf-qft}, the point counts mod $p$ match the eigenvalues
of the newform of weight $4$ and level $17$.  
More precisely, for $5 \le p \le 200$,
the number of $\F_p$-points of $O$ is equal to
$p^3+5p^2-(11+3(-3/p))p+1-a_p$, where $a_p$ is the eigenvalue of $T_p$ on the
newform $17/1$.  Again, 
the singularities are quite complicated and the Cynk-Szemberg
resolution needed to prove this true for all $p$
could only be constructed after 
several preliminary blowups.
%% which I have not yet been able to perform.
As in the cases of levels $5$ and $6$, examples of products of
elliptic surfaces realizing the same newform have been
found (\cite[Sections 2.1--2.2]{meyer-diss}) and there is expected to be
a correspondence between these and the threefold constructed here.

\bibliography{art}

\end{document}